\documentclass[10pt,a4paper]{article}
\usepackage[utf8]{inputenc}
\usepackage[T1]{fontenc}
\usepackage[english]{babel}
\usepackage{csquotes}
\usepackage[backend=bibtex,doi=false,style=numeric,maxnames=6]{biblatex}
\usepackage{tikz}
\usetikzlibrary{arrows, matrix, positioning, decorations.shapes, decorations.markings, decorations.pathreplacing, shapes, plotmarks, calc}
\usepackage{tikz-cd}
\usepackage{microtype}
\usepackage{hyperref} 
\usepackage{url} 
\usepackage[margin=1.5in]{geometry}
\usepackage{amsthm, amsmath, amsfonts}

\theoremstyle{definition}
\newtheorem{definition}{Definition}

\newtheorem{example}[definition]{Example}

\theoremstyle{plain}
\newtheorem{theorem}[definition]{Theorem}
\newtheorem{lemma}[definition]{Lemma}
\newtheorem{proposition}[definition]{Proposition}
\newtheorem{corollary}[definition]{Corollary}
\theoremstyle{theorem}
\newtheorem*{Merk}{Remark}

\newcommand{\prop}[1]{\begin{proposition}#1\end{proposition}}
\newcommand{\teo}[1]{\begin{theorem}#1\end{theorem}}
\newcommand{\lem}[1]{\begin{lemma}#1\end{lemma}}
\newcommand{\cor}[1]{\begin{corollary}#1\end{corollary}}

\newcommand{\df}[1]{\begin{definition}#1\end{definition}}

\newcommand{\idf}[1]{\textbf{#1}}

\newcommand{\Cech}{Čech}
\newcommand{\ie}{i.e.\ }
\newcommand{\eg}{e.g.\ }
\newcommand{\etal}{\textit{et al.\ }}
\DeclareMathOperator{\Link}{lk}
\DeclareMathOperator{\link}{lk}
\DeclareMathOperator{\depth}{depth}

\newcommand{\iso}{\cong}
\newcommand{\id}{\mathrm{id}}
\DeclareMathOperator{\bottleneck}{d_\mathrm{B}}
\DeclareMathOperator{\HC}{HC}
\DeclareMathOperator{\PH}{PH}
\DeclareMathOperator{\Conv}{CH}
\newcommand{\figref}[1]{Figure~\ref{fig:#1}}

\newcommand{\exref}[1]{Example~\ref{ex:#1}}
\newcommand{\secref}[1]{Section~\ref{sec:#1}}
\newcommand{\alphamax}{\ensuremath{\alpha_{\mathrm{max}}}}
\newcommand{\setsuchthat}{\ensuremath{\, \mid \,}}

\tikzset{twosimp/.style={fill opacity=0.6,fill=gray,draw opacity=1.0}}

\makeatletter
\newif\ifoldtikz
\@ifpackagelater{tikz}{2013/10/01} 
{
  \oldtikzfalse
}
{
  \oldtikztrue
}
\makeatother

\newcommand{\convexpath}[2]{
  \ifoldtikz
  \convexpathold{#1}{#2}
  \else
  \convexpathnew{#1}{#2}
  \fi
}
\newcommand{\convexpathold}[2]{
  [   
  create hullcoords/.code={
    \global\edef\namelist{#1}
    \foreach [count=\counter] \nodename in \namelist {
      \global\edef\numberofnodes{\counter}
      \coordinate (hullcoord\counter) at (\nodename);
    }
    \coordinate (hullcoord0) at (hullcoord\numberofnodes);
    \pgfmathtruncatemacro\lastnumber{\numberofnodes+1}
    \coordinate (hullcoord\lastnumber) at (hullcoord1);
  },
  create hullcoords
  ]
  ($(hullcoord1)!#2!-90:(hullcoord0)$)
  \foreach [
  evaluate=\currentnode as \previousnode using \currentnode-1,
  evaluate=\currentnode as \nextnode using \currentnode+1
  ] \currentnode in {1,...,\numberofnodes} {
    let \p1 = ($(hullcoord\currentnode) - (hullcoord\previousnode)$),
    \n1 = {atan2(\x1,\y1) + 90},
    \p2 = ($(hullcoord\nextnode) - (hullcoord\currentnode)$),
    \n2 = {atan2(\x2,\y2) + 90},
    \n{delta} = {Mod(\n2-\n1,360) - 360}
    in 
    {arc [start angle=\n1, delta angle=\n{delta}, radius=#2]}
    -- ($(hullcoord\nextnode)!#2!-90:(hullcoord\currentnode)$) 
  }
}
\newcommand{\convexpathnew}[2]{
  [   
  create hullcoords/.code={
    \global\edef\namelist{#1}
    \foreach [count=\counter] \nodename in \namelist {
      \global\edef\numberofnodes{\counter}
      \coordinate (hullcoord\counter) at (\nodename);
    }
    \coordinate (hullcoord0) at (hullcoord\numberofnodes);
    \pgfmathtruncatemacro\lastnumber{\numberofnodes+1}
    \coordinate (hullcoord\lastnumber) at (hullcoord1);
  },
  create hullcoords
  ]
  ($(hullcoord1)!#2!-90:(hullcoord0)$)
  \foreach [
  evaluate=\currentnode as \previousnode using \currentnode-1,
  evaluate=\currentnode as \nextnode using \currentnode+1
  ] \currentnode in {1,...,\numberofnodes} {
    let \p1 = ($(hullcoord\currentnode) - (hullcoord\previousnode)$),
    \n1 = {atan2(\y1,\x1) + 90},
    \p2 = ($(hullcoord\nextnode) - (hullcoord\currentnode)$),
    \n2 = {atan2(\y2,\x2) + 90},
    \n{delta} = {Mod(\n2-\n1,360) - 360}
    in 
    {arc [start angle=\n1, delta angle=\n{delta}, radius=#2]}
    -- ($(hullcoord\nextnode)!#2!-90:(hullcoord\currentnode)$) 
  }
}

\graphicspath{{./}{plots/}}

\title{Approximating Persistent Homology in Euclidean Space Through Collapses}

\author{
  Magnus Botnan\\
  \texttt{botnan@math.ntnu.no}
  \and
  Gard Spreemann\\
  \texttt{spreeman@math.ntnu.no}
}
\begin{document}

\maketitle

\begin{abstract}
The \Cech{} complex is one of the most widely used tools in applied algebraic topology. 
Unfortunately, due to the inclusive nature of the \Cech{} filtration, the number of simplices grows
exponentially in the number of input points. A practical consequence is that computations may have to
terminate at smaller scales than what the application calls for. 

In this paper we propose two methods to approximate the \Cech{} persistence module. Both
are constructed on the level of spaces, \ie as sequences of simplicial complexes induced by nerves.
We also show how the bottleneck distance between such persistence modules can be understood
by how tightly they are sandwiched on the level of spaces. In turn, this implies the correctness of our approximation methods.

Finally, we implement our methods and apply them to some example point clouds in Euclidean space. 

\end{abstract}

\section{Introduction}
Topological data analysis in general, and persistent homology in
particular, have shown great promise as tools for analyzing real-world
data arising in the sciences. Examples of successful applications range from
image analysis~\cite{bildeklein, joseklein}, to cancer research~\cite{puppekreft}, 
virology~\cite{viruslol} and sensor networks~\cite{deSilva_Ghrist_AGT:2007}.

Central to persistent homology are standard constructions for
recovering the homology of an underlying topological space from a
finite sample set, chiefly the \Cech{} and Vietoris--Rips
complexes. Unfortunately, due to the inclusive nature of their
filtrations, the number of simplices grows exponentially in the number
of sample points. This may be unfortunate as simplices added at small
scales may contribute little to homology at larger, possibly more
interesting, scales.

An extreme example may be a constant region in a measurement signal
(perhaps from faulty equipment or downtime) under time-delay
embedding~\cite{perea2013sliding}. In such a case, a large proportion
of the point cloud may lie in, say, a dense lump of $N$ points that
contributes nothing to the cloud's overall homology, yet introduces
$\binom{N}{k+1}$ $k$-simplices in the complex from an early scale.

Preprocessing of the point cloud may sometimes rectify the situation,
but such schemes are often decidedly ``off-line'' in the sense that
they require a one-off decision about which sparsifications to
effectuate ahead of persistence computations. We propose more
``on-line'' methods wherein a decision to attempt a simplification of
the simplicial complex may be made at any time during computations
when it is deemed necessary. The simplification operation itself
requires only that the point cloud comes supplied with its complete
linkage hierarchical clustering, which may be computed ahead of time
once and for all, or the computation of nets.

\subsection{Contributions}
The well-known Nerve lemma~\cite{citeulike:1282830} allows one to capture the topology
of a continuous space using discrete structures. However, the lemma works under the
assumption of a good cover, \ie a cover wherein every finite intersection of covering sets is contractible. 
This means that whenever we have a parametrized sequence of good covers, connected by maps of 
covers, the persistence diagram captured by the nerves equals the persistence diagram computed
by singular homology on the level of spaces. 

A central result in this paper is a way to bound
the bottleneck distance between these two persistence diagrams when the covers are not
necessarily good. Using this result we provide an approximation to the \Cech{} persistence module built on a finite sample from Euclidean space. 
The method enjoys several favorable properties: it approximates the \Cech{} persistence module
with provable error bounds and allows for size reduction on a heuristic basis, \ie only when the complex becomes
too large to store. Unfortunately, computing the weights of the simplices turns out to be expensive,
making it inapplicable in most settings. To mend this we propose an easy to compute approximation
which performs surprisingly well on real data sets. Using our aforementioned result we also show that
the net-tree construction as introduced by Sheehy~\cite{Sheehy13} and Dey \textit{et al.}~\cite{dey:2012:simplicialmaps} works well for the \Cech{} complex in
Euclidean space. This approach enjoys very powerful theoretical bounds, \eg a linear growth in
the number of simplices as a function of sampled points. In practice, however, it is difficult to 
prevent the complex from growing too large. 

Having implemented an algorithm to compute persistence
diagrams of simplicial complexes connected by simplicial maps we conclude the paper by
applying our approximations to a variety of point samples in Euclidean space. 

To the best of our knowledge, this is the first paper where
persistence computations are performed on simplicial complexes
connected by more general simplicial maps than inclusions.

\subsection{Outline}
In Section \ref{sec:back} we review background material and Dey \textit{et al.}'s algorithm~\cite{dey:2012:simplicialmaps} for computing persistent
homology of simplicial complexes connected by simplicial maps. In particular, we introduce
the concept of sequences of covers, and in Section \ref{sec:hocolim} we give a homotopy colimit
argument which relates the persistence module associated to a sequence of covers to that formed by the covering sets on the level of spaces. This relation is used in Section \ref{sec:sandwich}
to prove a sandwich type theorem for sequences of covers. We give two
approaches to approximating the \Cech{} persistence module in Section \ref{sec:cechrips}. The paper
concludes with Section \ref{sec:experiments} where we compute the persistence diagrams of 
example point clouds in Euclidean space using the aforementioned approximations. 

\subsection{Related work}
In low-dimensional Euclidean space the alpha complex~\cite{alfa} offers a memory efficient way to 
compute the persistence diagrams of a point cloud. Unfortunately, the number of simplices
grows exponentially in the ambient dimension, making it inefficient in high-dimensional
space. The witness complex~\cite{vitne} is a simplicial complex built on a subset of the sample, called landmarks. Unfortunately, the persistence diagrams of the associated filtration may depend heavily on the choice of landmarks.
 Sheehy~\cite{Sheehy13} and later
Dey \textit{et al.}~\cite{dey:2012:simplicialmaps} approximate the Vietoris--Rips complex using net-trees,
and Kerber and Sharathkumar~\cite{KerberS13} arrive at similar results for the \Cech{} complex in Euclidean space 
using quadtrees. Our constructions in Section~\ref{sec:linsize} is an adaption on the work of 
Dey \textit{et al.}~\cite{dey:2012:simplicialmaps} to 
the \Cech{} complex in Euclidean space. The construction in Section \ref{sec:nongood} can be viewed
as a particular type of a graph induced complex~\cite{dey:gic}. Chazal and Oudot~\cite{Chazal:2008:TPR:1377676.1377719} 
prove the results in \secref{hocolim} for the case where all the simplicial maps are inclusions. 

Recent research~\cite{dlot14, zom:tidy, vidit14} provides
methods to reduce the size of simplicial complexes after being stored, e.g.\ to provide faster persistence
computations. Such reductions are not discussed in this paper as we seek to compute persistence diagrams
of point clouds whose filtered complexes are too large to be stored to begin with.

\section{Background material} \label{sec:back} 

In this section we survey prerequisite background material and fix notation. 
We assume familiarity with basic concepts from algebraic topology, and basic
knowledge of persistent homology. For introductions see~\cite{citeulike:1282830} 
and~\cite{harer}, respectively. 

Throughout the paper, all simplicial complexes are assumed to be
finite and unoriented. A simplex is considered a set of vertices, and we write
a $k$-simplex $\{i_0,\dotsc,i_k\}$ as $[i_0,\dotsc,i_k]$. For
a simplicial complex $K$, we will denote its geometric realization by
$|K|$. Moreover, if $f: K\to L$ is a simplicial map between simplicial
complexes, then $|f|: |K| \to |L|$ denotes the continuous map between
their geometric realizations defined by $f$ on the vertices and
extended linearly using barycentric coordinates. The $p$-th singular
homology vector space of a topological space $X$ with coefficients in the field
$\mathbb{Z}_2$ will be denoted by $H_p(X)$, and for a continuous map
$f: X\to Y$ we denote its induced map on homology by $f_*: H_p(X) \to
H_p(Y)$. When $X = |K|$ is the geometric realization of a simplicial
complex, we will make no distinction between the $p$-th simplicial
homology vector space of $K$ and the $p$-th singular homology vector space of
$|K|$. Cohomology vector spaces over $\mathbb{Z}_2$ are similarly
denoted by $H^p(X)$.

A collection of open sets $\mathcal{U} = \{U_i \setsuchthat i\in I\}$ indexed by a finite set $I$ is said to be a
\idf{(finite) cover} of $\cup_{i\in I} U_i$. The \idf{nerve}
$N\mathcal{U}$ of the cover $\mathcal{U}$ is the simplicial complex
with vertex set $I$ and a $k$-simplex $[i_0, \ldots, i_k]\in
N\mathcal{U}$ if $U_{i_0} \cap\cdots\cap U_{i_k} \neq \emptyset$. Let
$\mathcal{U} = \{ U_i \setsuchthat i\in I\}$ and $\mathcal{V} = \{ V_j
\setsuchthat j \in J\}$ be covers of topological spaces $U\subseteq
V$. A map of sets $F: I\to J$ is said to be a \idf{map of covers} if
$U_{i}\subseteq V_{F(i)}$ for all $i\in I$. It is easy to check that
$F$ extends to a simplicial map $F: N\mathcal{U} \to N\mathcal{V}$
between the nerves of the covers. By a \idf{sequence of covers} we
will mean a collection of covers $\{\mathcal{U}(\alpha)\setsuchthat \alpha\in
  A\subset[0,\infty)\}$, each indexed respectively by $I(\alpha)$,
together with maps of covers $F^{\alpha, \alpha^\prime} : I(\alpha) \to
I(\alpha^\prime)$ such that $F^{\alpha,\alpha} = \id$ and
$F^{{\alpha},{\alpha^{\prime\prime}}} =
F^{{\alpha^\prime},{\alpha^{\prime\prime}}}\circ
F^{{\alpha},{\alpha^\prime}}$ for all $\alpha^{\prime\prime} \geq
\alpha^\prime \geq \alpha$. Such a sequence will be denoted by a pair
$(\mathcal{U}, F)$. Similarly, for any sequence of covers we have an
induced \idf{sequence of nerves} which will be denoted by
$(N\mathcal{U}, F)$.

\subsection{Persistence modules}
A \idf{persistence module} $\mathbb{V}$ over $A\subseteq\mathbb{R}$ is
a collection of $\mathbf{k}$-vector spaces $\{V(\alpha) \setsuchthat \alpha\in A \}$
and linear maps $v^{\alpha,{\alpha^\prime}}: V(\alpha) \to
V(\alpha^\prime)$ for all $\alpha\leq\alpha^\prime$ such that
$v^{\alpha,\alpha} = \id$ and $v^{{\alpha},{\alpha^{\prime\prime}}} =
v^{{\alpha^\prime},{\alpha^{\prime\prime}}}\circ
v^{{\alpha},{\alpha^\prime}}$. The \idf{direct sum} of two persistence
modules $\mathbb{U}$ and $\mathbb{W}$, both indexed over the same set,
is the persistence module $\mathbb{V} = \mathbb{U}\oplus\mathbb{W}$
where $V(\alpha) = U(\alpha)\oplus W(\alpha)$ and
$v^{\alpha,{\alpha^\prime}} = u^{\alpha,{\alpha^\prime}}\oplus
w^{\alpha,{\alpha^\prime}}$. We say that $\mathbb{V}$ is
\idf{indecomposable} if the only decompositions of $\mathbb{V}$ are
the trivial decompositions $0\oplus\mathbb{V}$ and $\mathbb{V}\oplus
0$.
\begin{definition}
  Let $J\subseteq A$ be an interval, \ie if $s,t\in J$ and $s<r<t$
  then $r\in J$. The \idf{interval module over $J$} is the persistence
  module $\mathbb{I}^J$ defined by
  \begin{align*}
    I^J(\alpha) = 
    \begin{cases}
      \mathbf{k} & \text{if } \alpha\in J\\
      0 & \text{otherwise}
    \end{cases}
  \end{align*}
  and $i^{\alpha,{\alpha^\prime}} = \id:I^J(\alpha)\to I^J(\alpha')$ whenever $\alpha\leq\alpha^\prime \in J$ and 0 otherwise.
\end{definition}

It is not difficult to show that $\mathbb{I}^J$ is
indecomposable, and the Krull--Remak--Schmidt--Azumaya
theorem~\cite{azumaya1950} tells us that if
\begin{align*}
  &\mathbb{V}\iso\bigoplus_{l\in L} \mathbb{I}^{J_l}  &\mathbb{V}\iso\bigoplus_{m\in M}\mathbb{I}^{K_m},
\end{align*}
then there is a bijection $\sigma: L\to M$ such that $J_l =
K_{\sigma(l)}$ for all $l\in L$. So whenever $\mathbb{V}$ admits such
a decomposition we can characterize it by the multiset $\{J_l
\setsuchthat l \in L\}$ of intervals called the \idf{persistence
  diagram} $\mathbf{D}(\mathbb{V})$ of $\mathbb{V}$. An
interval $(b,d)\in\mathbf{D}(\mathbb{V})$ represents a \idf{feature}
of $\mathbb{V}$ with \idf{birth} and \idf{death} time $b$ and $d$,
respectively. A persistence diagram is usually
depicted as a collection of points in $(\mathbb{R}\cup\{\pm\infty\})^2$. A recent theorem by
Crawley-Boevey~\cite{2012arXiv1210.08192C} asserts that $\mathbb{V}$
admits a decomposition into interval modules if $V_\alpha$ is
finite-dimensional for all $\alpha\in\mathbb{R}$. For an example of a
persistence module which does not admit an interval decomposition, see~\cite{2012arXiv1207.3674C}.

To every sequence of covers $(\mathcal{U}, F)$ we have an associated persistence
module $(H_p(N\mathcal{U}), F_*)$ with vector spaces
$\{H_p(N\mathcal{U}(\alpha)) \setsuchthat \alpha\in A\subseteq [0,\infty) \}$ and
maps $(F^{\alpha,{\alpha^\prime}})_*$. As the covers are finite, all
the homology vector spaces will have finite dimension, and thus the persistence diagrams
are well-defined. In particular, if $P\subseteq M$ is a finite set of
points in a metric space $M$, and $B(p; \alpha)$ is the open ball
of radius $\alpha$ centered at $p$, we get a sequence of covers by
defining $B(p;0) = \{p\}$, $\mathcal{U}(P;\alpha) = \{B(p; \alpha) \setsuchthat p \in P\}$
and $F=\id$. The induced sequence of nerves is known as the
\idf{\Cech{} filtration} and the associated persistence module is the
\idf{\Cech{} persistence module}. In the
remainder of this paper $\mathcal{C}(P; \alpha)$ denotes the nerve
of the \Cech{} filtration of $P$ at scale $\alpha$. 

Another popular construction is the Vietoris--Rips complex
$\mathcal{R}(P;\alpha)$ which is defined as the largest simplicial
complex with the same $1$-skeleton as $\mathcal{C}(P; \alpha)$. By
definition, it follows that $\mathcal{C}(P; \alpha) \subseteq
\mathcal{R}(P; \alpha)$, and for $P\subseteq\mathbb{R}^n$, it is also
true that $\mathcal{R}(P; \alpha) \subseteq \mathcal{C}(P;
\sqrt{2}\alpha)$~\cite{deSilva_Ghrist_AGT:2007}. 

\subsection{Metrics and approximations}
Let $\Delta$ denote the multiset of all pairs
$(x,x)\in(\mathbb{R}\cup\{\pm\infty\})^2$, each with countably infinite
multiplicity. A \idf{partial matching} between two persistence
diagrams $D$ and $D'$ is a bijection $\gamma: B\cup\Delta \to
B'\cup\Delta$, and we denote all such by $\Gamma(D,D')$.

The following defines a metric on persistence diagrams:
\begin{definition}
  The \idf{bottleneck distance} between two persistence diagrams
  $B$ and $B'$ is 
  \begin{equation*}
    \bottleneck(B,B') = \inf_{\gamma\in\Gamma(D,D')}\sup_{(b,d)\in B} ||(b,d)-\gamma((b,d))||_\infty
  \end{equation*}
  where
  \begin{equation*}
    ||(b_1,d_1)-(b_2,d_2)||_\infty = \max(|b_1-b_2|, |d_1-d_2|).
  \end{equation*}
\end{definition}
The theory of
\emph{interleavings}~\cite{Chazal:2009:PPM:1542362.1542407} offers a
generalization of the bottleneck distance to persistence modules that
do not admit a decomposition into indecomposables. Importantly, if
there exists an $\epsilon$-interleaving between two persistence
modules, then their bottleneck distance is at most $\epsilon$. In this
paper we adopt the conventions of \cite{Sheehy13, KerberS13} and use a
slight reformulation of the ordinary theory of interleavings.

\begin{definition}
  Two persistence modules $\mathbb{U}$ and $\mathbb{V}$ indexed over
  $[0,\infty)$ are said to be $c$-approximate if there exist a
  constant $c\geq 1$ and two families of homomorphisms $\{\phi_\alpha
  : U(\alpha)\to V(c\alpha)\}_{\alpha\geq 0}$ and $\{\psi_\alpha:
  V(\alpha) \to U(c\alpha)\}_{\alpha\geq 0}$ such that the following
  four diagrams commute for all $\alpha\leq \alpha^\prime$:
  \begin{equation*}
    \begin{tikzcd}
      U(\frac{\alpha}{c}) \arrow{dr}\arrow{rrr} &~&~&U(c\alpha^\prime)& U(c\alpha) \arrow{r}
      &U(c\alpha^\prime) \\
      ~& V(\alpha)\arrow{r}
      & V(\alpha^\prime)\arrow{ur}&  V(\alpha) \arrow{r}\arrow{ur} 
      & V(\alpha^\prime)\arrow{ur} & ~\\
      ~&U(\alpha)\arrow{r} & U(\alpha^\prime) \arrow{dr} & U(\alpha)\arrow{r}\arrow{dr} & U(\alpha^\prime)\arrow{dr}& ~\\
      V(\frac{\alpha}{c})\arrow{ru}\arrow{rrr}&~&~&V(c\alpha^\prime) & V(c\alpha)\arrow{r} & V(c\alpha^\prime)
    \end{tikzcd}
  \end{equation*}
  \label{def:capprox}
\end{definition}
The following theorem is immediate from the theory of interleavings~\cite{Chazal:2009:PPM:1542362.1542407}.
\begin{theorem}
  If $\mathbb{U}$ and $\mathbb{V}$ are $c$-approximate, then their
  bottleneck distance is bounded by $\log c$ on the $\log$-scale.
  \label{teo:bottdistance}
\end{theorem}

The above result can be seen as a general version of the relationship
between the \Cech{} and Vietoris--Rips filtrations. Indeed, while the
bottleneck distance between their persistence diagrams may be
arbitrarily large, the inclusions
\begin{equation*}
  \mathcal{C}(P;\alpha) \subseteq \mathcal{R}(P;\alpha) \subseteq \mathcal{C}(P;\sqrt{2}\alpha)  
\end{equation*}
ensure that a feature $(b,d)$ in the Vietoris--Rips persistence module
is also a feature in the \Cech{} persistence module if
$d-b\geq\sqrt{2}b$, and vice versa.

\subsection{Computing persistent homology using
  annotations} \label{sec:computing} 

Many widely implemented and used algorithms for computing persistent
homology assume that the maps in the persistence module are induced by
inclusions of simplicial complexes, \ie that the underlying sequence
is a filtration. As shall become clear, we will need to compute in the
setting of general simplicial maps.

\begin{definition}
  A surjective simplicial map $f:K\to K'$ with the property that there
  exist distinct $[a],[b]\in K$ such that
  \begin{align*}
    f(\sigma) = \begin{cases}
      \sigma\setminus\{b\} &\quad\text{ if $a,b\in\sigma$ } \\
      \{a\}\cup\sigma\setminus\{b\} &\quad\text{ if $a\notin\sigma$, $b\in\sigma$ }  \\
      \sigma &\quad\text{ otherwise}
      \end{cases}
  \end{align*}
  is called an \idf{edge contraction of $[a,b]$ to $[a]$}. Simplices
  $\sigma,\sigma'\in K$ are called \idf{mirror simplices} (for $f$) if
  $f(\sigma)=f(\sigma')$.
\end{definition}
We will often refer to an edge contraction like that above by
$[a,b]\mapsto[a]$. Since up to isomorphism any simplicial map $K\to
K'$ decomposes into a finite sequence of inclusions and edge
contractions, we only need to deal with those two types and adjust the
persistence module indices accordingly to reflect the addition of
extra maps. Likewise, as is normal, we decompose inclusions into ones
of the form $K\to K\cup \{\sigma\}$ and refer to these as ``adding a
simplex $\sigma$''.

We will use Dey \textit{et al.}'s method of \emph{persistence
  annotations}~\cite{dey:2012:simplicialmaps} to compute (the
persistence diagrams of) persistence modules with simplicial maps, and
now quickly review their algorithm and our implementation details.

The method of annotation tracks homology with $\mathbb{Z}_2$
coefficients across a persistence module by storing the value of all
cohomology generators at each simplex and updating these
``annotations'' to reflect the inclusion of a simplex or the
contraction of an edge. Care should be taken to notice a slight
difference in terminology: our definition of annotations reflects
Dey's \emph{valid annotations}.
\begin{definition}
  An \idf{annotation} for a simplicial complex $K$ is a linear map
  $\Phi_p : C_p(K)\to\mathbb{Z}_2^n$ with the property that
  \begin{equation*}
    \varphi_1 = [c\mapsto\Phi_p(c)_1],\dotsc, \varphi_n = [c\mapsto\Phi_p(c)_n]
  \end{equation*}
  is a basis for $H^p(K)$. Here $\Phi_p(c)_i$ denotes
  the $i$'th component of $\Phi_p(c)\in\mathbb{Z}_2^n$.
\end{definition}

A key observation is the following: the persistent homology
of a sequence of simplicial complexes can be obtained by dualizing on the level of chains
and taking cohomology. This is true since when working over $\mathbb{Z}_2$ (or any field),
the map $\alpha:H^p(K) \to {\rm Hom}(H_p(K), \mathbb{Z}_2)$ defined by
$\alpha([f])([c]) = f(c)$ is an isomorphism. Thus, intervals
in persistent cohomology are dual to intervals in persistent homology. Therefore,
we shall interchangeably speak of a homology class born at persistence index $i$
as a cohomology class in the opposite direction dying at persistence index $i$. 

By storing the value of $\Phi_p$ at each $p$-simplex, that simplex'
contribution to the (co)homology vector space is known and so allows
us to only make changes to homology near the site of a
contraction. This ``locality'' of the changes introduced by an edge
contraction is summarized in the following
definition~\cite{dey:1999:contraction}, proposition~\cite{attali:2012}
and lemmas.
\begin{definition}
  The \idf{link} of a simplex $\sigma$ in a simplicial complex $K$ is the set
  \begin{equation*}
    \Link_K\sigma = \{\tau\setminus\sigma \setsuchthat \sigma\subseteq\tau\in K\}.
  \end{equation*}
  An edge $[a,b]\in K$ satisfies the \idf{link condition} if
  $\Link_K[a]\cap\Link_K[b] = \Link_K[a,b]$.
\end{definition}
When the simplicial complex in question is clear, we shall simply
write $\link$ for $\link_K$. 
\begin{proposition} \label{prop:lkcond} The contraction $f:K\to K'$ of
  an edge that satisfies the link condition induces a homotopy
  equivalence $|f|:|K|\to |K'|$, and hence an isomorphism
  $f_\ast:H_\ast(K)\to H_\ast(K')$.
\end{proposition}
\begin{lemma} 
  If $[a,b]\in K$, then
  $\link_K[a,b]\subseteq\link_K[a]\cap\link_K[b]$.
\end{lemma}
\begin{proof}
  Suppose $\eta\in\link[a,b]$. Then there exists a $\tau\in K$ with
  $[a,b]\subseteq\tau$ and $\eta=\tau\setminus[a,b]$. Since $K$ is a
  simplicial complex, it also contains $\tau'=\tau\setminus[a]$ and
  $\tau''=\tau\setminus[b]$. We have $[b]\subseteq\tau'$ and
  $\eta=\tau'\setminus[b]$, so $\eta\in\link[b]$. The same argument
  using $\tau''$ gives that $\eta\in\link[a]$.
\end{proof}
For the following lemma we shall write
$L_K(a,b)=(\link_{K}[a]\cap\link_{K}[b])\setminus\link_{K}[a,b]$.
\begin{lemma} \label{lem:lkcond} If $\eta \in L_K(a,b)$, then $K'=
  K\cup\{\eta\cup[a,b]\}$ is also a simplicial complex, and moreover
  $L_{K'}(a,b) = L_K(a,b)\setminus\{\eta\}$.
\end{lemma}
\begin{proof}
  Observe that $[a,b]\not\subseteq\eta$. $K'$ is still a simplicial complex, as all faces of $\eta\cup[a,b]$ are present in $K$ by the
  assumption that $\eta\in L_K(a,b)$. Note that by definition
  \begin{align*}
    &\link_{K'}[a] = \link_K[a]\cup\{\eta\cup[b]\}& 
    &\link_{K'}[b] = \link_K[b]\cup\{\eta\cup[a]\},&
  \end{align*}
  so $\link_{K'}[a]\cap\link_{K'}[b]=\link_{K}[a]\cap\link_{K}[b]$. It
  also follows from the definition that
  \begin{equation*}
    \link_{K'}[a,b] = \link_{K}[a,b]\cup\{\eta\},
  \end{equation*}
  so $L_{K'}(a,b) = L_K(a,b)\setminus\{\eta\}$.
\end{proof}
In summary, we see that to contract an edge we only need to change the
simplicial complex in the vicinity of that edge.

Suppose 
\begin{equation*}
  K=(K_0\xrightarrow[]{f_0}{}K_1 \xrightarrow[]{f_{1}} \dotsm \xrightarrow[]{f_{m-1}} K_m)
\end{equation*}
is a sequence of simplicial complexes (with the $f_i$'s simplicial
maps) whose persistence module
\begin{equation*}
  H_\ast(K)=(H_\ast(K_0)\xrightarrow[]{(f_0)_\ast}{}H_\ast(K_1) \xrightarrow[]{(f_{1})_\ast} \dotsm \xrightarrow[]{(f_{m-1})_\ast} H_\ast(K_m))
\end{equation*}
has been computed, and
write $\Phi_p^i$ for the annotation of $H^p(K_i)$ and $n$ for its
dimension. To compute the persistence module of
\begin{equation*}
  K'=(K_0\xrightarrow[]{f_0}{} \dotsm \xrightarrow[]{f_{m-1}} K_m \xrightarrow[]{f_{m}} K_{m+1}),
\end{equation*}
there are four cases to handle:
\begin{enumerate}
  \item $f_{m}$ adds a single $p$-simplex $\sigma$, and\dots
    \begin{enumerate}
    \item \label{casegenborn} $\Phi_{p-1}^m(\partial\sigma)=0$. This
      corresponds to a generator of $H_p(K')$ being born at
      persistence index $m+1$, or equivalently to a generator of
      $H^p(K')$ dying at $m$ going left (see Proposition~5.2
      in~\cite{dey:2012:simplicialmaps}). Define
      $\Phi_p^{m+1}:C_p(K_{m+1})\to\mathbb{Z}_2^{n+1}$ by
      \begin{align*}
        \Phi_p^{m+1}(\tau) = \begin{cases}
          (\Phi_p^m(\tau)_1,\dotsc,\Phi_p^m(\tau)_n, 0) &\quad\text{ if } \tau\neq\sigma \\
          (0,\dotsc,0, 1) &\quad\text{ if } \tau = \sigma
        \end{cases}
      \end{align*}
      and extending linearly. In other dimensions $q\neq p$, we set
      $\Phi_q^{m+1}=\Phi_q^m$.
    \item \label{casegenkilled}
      $\Phi_{p-1}^m(\partial\sigma)_{i_1}=\dotsm=\Phi_{p-1}^m(\partial\sigma)_{i_l}=1$
      for some $l\geq 1$. In this case $\sigma$ kills a class in
      $H_{p-1}(K')$ at $m+1$, or equivalently gives birth to one of the
      generators $\varphi_{i_1},\dotsc,\varphi_{i_l}$ of $H^{p-1}(K')$
      in the reverse direction (see Proposition~5.2 in~\cite{dey:2012:simplicialmaps}). We kill the youngest homology class,
      say the one numbered $u$ (so $\varphi_u$ is born in the reverse
      direction). Note that $\gamma:K_{m+1}\to\mathbb{Z}_2^n$ defined by
      \begin{align*}
        \gamma(\tau) = \begin{cases}
          \Phi_{p-1}^m(\tau) + \Phi_{p-1}^m(\partial\sigma) &\quad\text{ if } \Phi_{p-1}^m(\tau)_u = 1 \\
          \Phi_{p-1}^m(\tau) &\quad\text{ otherwise}
          \end{cases}
      \end{align*}
      has $0$ in component $u$ of all its values. Define
      $\Phi_{p-1}^{m+1}:C_{p-1}(K_{m+1})\to\mathbb{Z}_2^{n-1}$ as
      $\gamma$ with the $u$-th component removed, and extend linearly.
      In other dimensions $q\neq p-1$, we set
      \begin{align*}
        \Phi_q^{m+1}(\tau) = \begin{cases}
          (0,\dotsc,0) &\quad\text{ if } \tau = \sigma, q=p\\
          \Phi_q^m(\tau) &\quad\text{ otherwise}.
        \end{cases}
      \end{align*}
    \end{enumerate}
  \item $f_{m}$ contracts $[a,b]$ to $[a]$, and\dots
    \begin{enumerate}
    \item \label{caselinkcond} $[a,b]$ satisfies the link
      condition. Let 
      \begin{equation*}
        M_{p-1}=\{\sigma \in K_m \setsuchthat \dim\sigma = p-1 , a\in\sigma \text{ and }\sigma\text{ has a mirror under } f_m\},
      \end{equation*}
      and note that to any $\tau\in M_{p-1}$, there is a unique
      $g_\tau\in K_m$ with $\tau\subseteq g_\tau$, $\dim g_\tau = p$ and
      $[a,b]\subseteq g_\tau$. Define $\Phi_p^{m+1}$ on the
      $p$-simplices of $K_{m+1}$ by
      \begin{align*}
        \Phi_p^{m+1}(\sigma) = \Phi_p^m(\sigma) + \underset{\substack{\sigma\supseteq\tau\in M_{p-1}}}{\sum}\Phi_p^m(g_\tau),
      \end{align*}
      noting that the sum may be empty. This corresponds to Dey's
      ``annotation transfers'' --- see Proposition 4.4 and 4.5
      of~\cite{dey:2012:simplicialmaps} for a more detailed
      explanation.
    \item $[a,b]$ does not satisfy the link
      condition. Lemma~\ref{lem:lkcond} tells us which simplices to
      add, repeatedly hitting the cases~\ref{casegenborn}
      and~\ref{casegenkilled}, until the link condition becomes
      fulfilled\footnote{This must happen after a finite number of
        steps since Lemma~\ref{lem:lkcond} shows that the size of
        $(\link[a]\cap\link[b])\setminus\link[a,b]$ is reduced by one
        every time one of the new simplices is added. Moreover, one
        can in practice expect the number of simplices added to be
        small compared to the size of $K_m$ since only cofaces of
        $[a,b]$ are added.}. Afterwards contracting $[a,b]$ is handled
      by case~\ref{caselinkcond}. Some bookkeeping is of course
      required if one wants to consider the potentially many homology
      changes from the inclusions as occurring at persistence index
      $m+1$.
    \end{enumerate}
\end{enumerate}
Dey \etal show in~\cite{dey:2012:simplicialmaps} (Proposition~5.1)
that $\Phi_\ast^{m+1}$ as constructed above is an annotation for
$H^\ast(K_{m+1})$. With $K_0=\emptyset$ and the associated empty
annotation $\Phi_\ast^0$, then, the above is a correct algorithm for
computing persistent homology.

\subsubsection{Some implementation details}
As suggested in~\cite{boissonnat:2012:efficient,
  boissonnat:2012:simplextree}, the \emph{simplex tree} is a data
structure that is well-suited for storing the simplicial complex in the above
algorithm.

A simplex tree is a \emph{trie} (also called a prefix tree), which is
a tree $T$ that stores a simplicial complex $K$ whose vertices $V$
have a total ordering $\leq$ by the following rules:
\begin{itemize}
  \item $T$ contains a distinguished root.
  \item Every non-root node $n\in T$ carries the data of a label
    $L(n)\in V$. The root is labelled by a distinguished symbol, say
    $\ast$, and we extend the ordering to $\ast < v$ for all $v\in
    V$ to ease notation.
  \item Nodes have zero or more children.
  \item If $n$ is a child of $p$, then $L(n)>L(p)$.
  \item If $n$ and $m$ both are children of $p$, then $L(n)\neq L(m)$.
\end{itemize}
The simplicial complex $K$ to be encoded corresponds to all paths to
the root of $T$, and we write $S(n)\in K$ for the simplex
corresponding to the path from $n\in T$.  We will also refer to the
root having \idf{depth} $0$, and in general a node as having depth
$k+1$ if its parent has depth $k$. Thus $\depth(n) = \dim S(n) + 1$.

In terms of implementation, every node holds a pointer to its parent
and a dictionary\footnote{A dictionary is here any data structure with
  logarithmic lookup time complexity for keys.} of pointers to its
children, keyed on their labels. Furthermore, we augment the tree by
adding to each node a ``cousin pointer'': We call $m$ a \idf{cousin} of $n$
if $\depth(m)=\depth(n)$ and $L(m) = L(n)$. Every node holds a pointer
to one of its cousins in such a way that they form a cyclic linked
list that visits every cousin at the same depth precisely once (per
cycle). In addition, an arbitrary representative of each such cyclic
linked list is maintained in a dictionary keyed on labels and depths.

\figref{simplextree} shows an example of the basic part of a simplex
tree, along with an example of annotations (intermediate data
structures are dropped from the figure, and annotations are attached
directly to the simplices for ease of visualization).

\begin{figure}[htpb]
  \centering
  \begin{tikzpicture}
    \begin{scope}[xshift=0cm, yshift=0cm]      
      \begin{scope}[xshift=0cm, yshift=0cm]
        \coordinate (1) at (0,0);
        \coordinate (2) at (2,0);
        \coordinate (3) at (2,2);
        \coordinate (4) at (0,2);
        \coordinate (5) at (1,1);
        \draw (5) -- (4) -- (3) -- cycle;
        \draw (1) -- (2) -- (5) -- cycle;
        \draw[twosimp] (1) -- (5) -- (4) -- cycle;
        \draw[twosimp] (2) -- (3) -- (5) -- cycle;
        
        \node () at (1.west) [anchor=east] {$1$};
        \node () at (2.east) [anchor=west] {$2$};
        \node () at (3.east) [anchor=west] {$3$};
        \node () at (4.east) [anchor=east] {$4$};
        \node () at (5.north) [anchor=south] {$5$};
       
        \draw (3) -- (4) node[midway,above,sloped,font=\tiny] {$0010$};
        \draw (1) -- (2) node[midway,below,sloped,font=\tiny] {$0001$};
        \draw (4) -- (5) node[midway,above,sloped,font=\tiny] {$0000$};
        \draw (3) -- (5) node[midway,above,sloped,font=\tiny] {$0000$};
        \draw (1) -- (4) node[midway,above,sloped,font=\tiny] {$0000$};
        \draw (2) -- (3) node[midway,below,sloped,font=\tiny] {$0000$};
        \draw (1) -- (5) node[midway,above,sloped,font=\tiny] {$0000$};
        \draw (2) -- (5) node[midway,above,sloped,font=\tiny] {$0000$};
      \end{scope}
      \begin{scope}[xshift=8cm, yshift=2cm, every node/.style={draw,rectangle,inner sep=3pt, minimum size=0.3cm}]
        \node[anchor=east] (root) at (-0.75,0) {$\ast$};
        \node[anchor=west] (4) at (-0.75,-0.75) {$4$};
        \node[anchor=east] (3) at (4.west) {$3$};
        \node[anchor=west] (5) at (4.east) {$5$};
        \node[anchor=east] (2) at (3.west) {$2$};
        \node[anchor=east] (1) at (2.west) {$1$};
        \draw[decorate,decoration=brace] (1.north west) -- (5.north east) coordinate[midway] (litter1) {};
        \draw (litter1) + (0, 2.5pt) -- (root.south);   

        \node[anchor=west] (12) at (-3.5,-1.5) {$2$};
        \node[anchor=west] (14) at (12.east) {$4$};
        \node[anchor=west] (15) at (14.east) {$5$};
        \draw[decorate,decoration=brace] (12.north west) -- (15.north east) coordinate[midway] (litter2);
        \draw (litter2) + (0, 2.5pt) -- (1.south);

        \node[anchor=west] (23) at (-1.7,-1.5) {$3$};
        \node[anchor=west] (25) at (23.east) {$5$};
        \draw[decorate,decoration=brace] (23.north west) -- (25.north east) coordinate[midway] (litter4);
        \draw (litter4) + (0, 2.5pt) -- (2.south);

        \node[anchor=west] (34) at (-0.75,-1.5) {$4$};
        \node[anchor=west] (35) at (34.east) {$5$};
        \draw[decorate,decoration=brace] (34.north west) -- (35.north east) coordinate[midway] (litter6);
        \draw (litter6) + (0, 2.5pt) -- (3.south);

        \node[anchor=west] (45) at (0.5,-1.5) {$5$};
        \draw[decorate,decoration=brace] (45.north west) -- (45.north east) coordinate[midway] (litter7);
        \draw (litter7) + (0, 2.5pt) -- (4.south);

        \node [anchor=east] (145) at (-2.5, -2.25) {$5$};
        \draw[decorate,decoration=brace] (145.north west) -- (145.north east) coordinate[midway] (litter3);
        \draw (litter3) + (0, 2.5pt) -- (14.south);

        \node [anchor=east] (235) at (-1, -2.25) {$5$};
        \draw[decorate,decoration=brace] (235.north west) -- (235.north east) coordinate[midway] (litter8);
        \draw (litter8) + (0, 2.5pt) -- (23.south);
      \end{scope}

    \end{scope}

    \begin{scope}[xshift=0cm, yshift=-3cm]
      \begin{scope}[xshift=0cm, yshift=0cm]
        \coordinate (1) at (0,0);
        \coordinate (2) at (2,0);
        \coordinate (3) at (2,2);
        \coordinate (4) at (0,2);
        \coordinate (5) at (1,1);
        \draw (5) -- (4) -- (3) -- cycle;
        \draw (1) -- (2) -- (5) -- cycle;
        \draw[twosimp] (1) -- (5) -- (4) -- cycle;
        \draw[twosimp] (2) -- (3) -- (5) -- cycle;
        \draw[twosimp] (1) -- (2) -- (5) -- cycle;
        
        \node () at (1.west) [anchor=east] {$1$};
        \node () at (2.east) [anchor=west] {$2$};
        \node () at (3.east) [anchor=west] {$3$};
        \node () at (4.east) [anchor=east] {$4$};
        \node () at (5.north) [anchor=south] {$5$};
        
        \draw (3) -- (4) node[midway,above,sloped,font=\tiny] {$0010$};
        \draw (1) -- (5) node[midway,above,sloped,font=\tiny] {$0000$};
        \draw (2) -- (5) node[midway,above,sloped,font=\tiny] {$0000$};
        \draw (4) -- (5) node[midway,above,sloped,font=\tiny] {$0000$};
        \draw (3) -- (5) node[midway,above,sloped,font=\tiny] {$0000$};
        \draw (1) -- (4) node[midway,above,sloped,font=\tiny] {$0000$};
        \draw (2) -- (3) node[midway,below,sloped,font=\tiny] {$0000$};
        \draw (1) -- (2) node[midway,below,sloped,font=\tiny] {$0000$};
      \end{scope}

      \begin{scope}[xshift=8cm, yshift=2cm, every node/.style={draw,rectangle,inner sep=3pt, minimum size=0.3cm}]
        \node[anchor=east] (root) at (-0.75,0) {$\ast$};
        \node[anchor=west] (4) at (-0.75,-0.75) {$4$};
        \node[anchor=east] (3) at (4.west) {$3$};
        \node[anchor=west] (5) at (4.east) {$5$};
        \node[anchor=east] (2) at (3.west) {$2$};
        \node[anchor=east] (1) at (2.west) {$1$};
        \draw[decorate,decoration=brace] (1.north west) -- (5.north east) coordinate[midway] (litter1) {};
        \draw (litter1) + (0, 2.5pt) -- (root.south);   

        \node[anchor=west] (12) at (-3.5,-1.5) {$2$};
        \node[anchor=west] (14) at (12.east) {$4$};
        \node[anchor=west] (15) at (14.east) {$5$};
        \draw[decorate,decoration=brace] (12.north west) -- (15.north east) coordinate[midway] (litter2);
        \draw (litter2) + (0, 2.5pt) -- (1.south);

        \node[anchor=west] (23) at (-1.7,-1.5) {$3$};
        \node[anchor=west] (25) at (23.east) {$5$};
        \draw[decorate,decoration=brace] (23.north west) -- (25.north east) coordinate[midway] (litter4);
        \draw (litter4) + (0, 2.5pt) -- (2.south);

        \node[anchor=west] (34) at (-0.75,-1.5) {$4$};
        \node[anchor=west] (35) at (34.east) {$5$};
        \draw[decorate,decoration=brace] (34.north west) -- (35.north east) coordinate[midway] (litter6);
        \draw (litter6) + (0, 2.5pt) -- (3.south);

        \node[anchor=west] (45) at (0.5,-1.5) {$5$};
        \draw[decorate,decoration=brace] (45.north west) -- (45.north east) coordinate[midway] (litter7);
        \draw (litter7) + (0, 2.5pt) -- (4.south);

        \node [anchor=east] (145) at (-2.5, -2.25) {$5$};
        \draw[decorate,decoration=brace] (145.north west) -- (145.north east) coordinate[midway] (litter3);
        \draw (litter3) + (0, 2.5pt) -- (14.south);

        \node [anchor=east] (235) at (-1, -2.25) {$5$};
        \draw[decorate,decoration=brace] (235.north west) -- (235.north east) coordinate[midway] (litter8);
        \draw (litter8) + (0, 2.5pt) -- (23.south);

        \node [anchor=east] (125) at (-3.25, -2.25) {$5$};
        \draw[decorate,decoration=brace] (125.north west) -- (125.north east) coordinate[midway] (litter9);
        \draw (litter9) + (0, 2.5pt) -- (12.south);

      \end{scope}
    \end{scope}

    \begin{scope}[xshift=0cm, yshift=-6cm]
      \begin{scope}[xshift=0cm, yshift=0cm]
        \coordinate (1) at (1,0);
        \coordinate (3) at (2,2);
        \coordinate (4) at (0,2);
        \coordinate (5) at (1,1);
        \draw (5) -- (4) -- (3) -- cycle;
        \draw[twosimp] (1) -- (4) -- (5) -- cycle;
        \draw[twosimp] (1) -- (3) -- (5) -- cycle;
        
        \node () at (1.west) [anchor=east] {$1$};
        \node () at (3.east) [anchor=west] {$3$};
        \node () at (4.east) [anchor=east] {$4$};
        \node () at (5.north) [anchor=south] {$5$};
        
        \draw (3) -- (4) node[midway,above,sloped,font=\tiny] {$0010$};
        \draw (1) -- (5) node[near end,above,sloped,font=\tiny] {$0000$};
        \draw (4) -- (5) node[midway,above,sloped,font=\tiny] {$0000$};
        \draw (3) -- (5) node[midway,above,sloped,font=\tiny] {$0000$};
        \draw (1) -- (4) node[midway,below,sloped,font=\tiny] {$0000$};
        \draw (1) -- (3) node[midway,below,sloped,font=\tiny] {$0000$};
      \end{scope}

      \begin{scope}[xshift=8cm, yshift=2cm, every node/.style={draw,rectangle,inner sep=3pt, minimum size=0.3cm}]
        \node[anchor=east] (root) at (-0.75,0) {$\ast$};
        \node[anchor=west] (4) at (-0.75,-0.75) {$4$};
        \node[anchor=east] (3) at (4.west) {$3$};
        \node[anchor=west] (5) at (4.east) {$5$};
        \node[anchor=east] (1) at (3.west) {$1$};
        \draw[decorate,decoration=brace] (1.north west) -- (5.north east) coordinate[midway] (litter1) {};
        \draw (litter1) + (0, 2.5pt) -- (root.south);   

        \node[anchor=west] (13) at (-2.5,-1.5) {$3$};
        \node[anchor=west] (14) at (13.east) {$4$};
        \node[anchor=west] (15) at (14.east) {$5$};
        \draw[decorate,decoration=brace] (13.north west) -- (15.north east) coordinate[midway] (litter2);
        \draw (litter2) + (0, 2.5pt) -- (1.south);

        \node[anchor=west] (34) at (-1,-1.5) {$4$};
        \node[anchor=west] (35) at (34.east) {$5$};
        \draw[decorate,decoration=brace] (34.north west) -- (35.north east) coordinate[midway] (litter6);
        \draw (litter6) + (0, 2.5pt) -- (3.south);

        \node[anchor=west] (45) at (0.25,-1.5) {$5$};
        \draw[decorate,decoration=brace] (45.north west) -- (45.north east) coordinate[midway] (litter7);
        \draw (litter7) + (0, 2.5pt) -- (4.south);

        \node [anchor=east] (145) at (-1.5, -2.25) {$5$};
        \draw[decorate,decoration=brace] (145.north west) -- (145.north east) coordinate[midway] (litter3);
        \draw (litter3) + (0, 2.5pt) -- (14.south);

        \node [anchor=east] (135) at (-2.5, -2.25) {$5$};
        \draw[decorate,decoration=brace] (135.north west) -- (135.north east) coordinate[midway] (litter8);
        \draw (litter8) + (0, 2.5pt) -- (13.south);

      \end{scope}
    \end{scope}

  \end{tikzpicture}
  \caption{A somewhat simplified simplex tree representation of a
    simplicial complex. Annotation values on the $1$-simplices are
    included for a persistence module in which the simplices are added
    in the order
    $\dotsc,[1,5],[4,5],[2,5],[3,5],[1,4],[2,3],[1,4,5],[2,3,5],[3,4],[1,2]$,
    leading up to the top row situation. To contract the edge $[1,2]$,
    the link condition must be fulfilled, requiring the inclusion of
    $[1,2,5]$ (middle row). The situation after contraction is shown
    in the bottom row.} \label{fig:simplextree}
\end{figure}
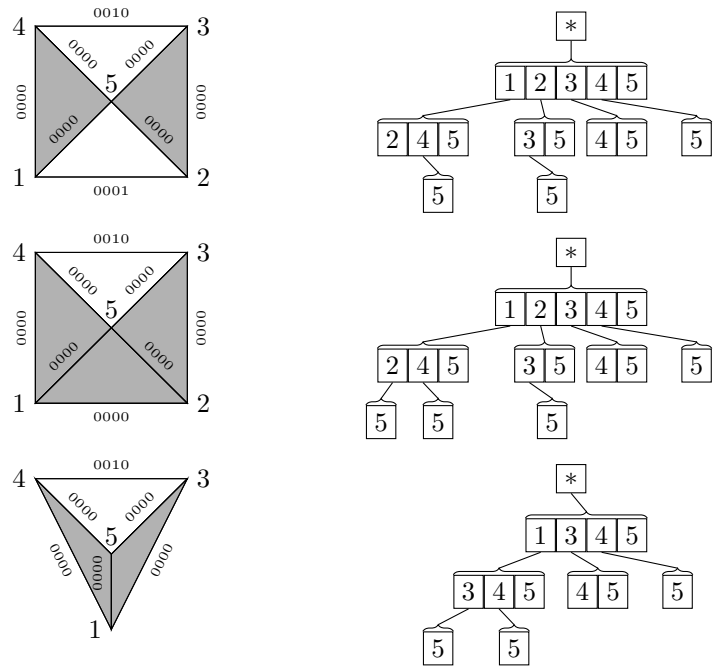

Boissonnat and Maria show that this data structure allows us to
efficiently insert and remove simplices, and compute their faces and
cofaces. For details, see~\cite{boissonnat:2012:simplextree}.

To tie the simplex tree to the annotations discussed earlier, we want
to associate to each node (\ie each simplex) its annotation
value. Since multiple simplices are likely to share the same
annotation value, we go by way of a union find structure. Each node
thus contains a pointer to a node in a forest, wherein each tree
represents an annotation value shared by multiple cohomologous
simplices. The root of each tree in the forest points to the actual
annotation value of the simplices pointing to nodes in that tree.

The annotation values themselves are also kept referenced in a
dictionary (keyed on the annotation values) for easy access and
updating as used in the algorithm outlined earlier.

\section{Persistent homology of sequences of
  covers} \label{sec:hocolim} 

In the following we assume that all covering sets are subsets of some
metric space and that every cover is finite. In particular, this means
that all our spaces are paracompact.  Moreover, the constructions in
this section can be seen as special cases of the much more general
construction of a homotopy colimit of a diagram of topological spaces.

To any open covering $\mathcal{U} = \{U_i \setsuchthat i\in I\}$ of an open set $U$ we 
assign a topological space $\Delta U_\mathcal{U}\subset |N\mathcal{U}|\times U$ defined 
as the disjoint union 
\begin{equation*}
  \bigsqcup_{S \in N\mathcal{U}}|S|\times \bigcap_{i\in S}U_i
\end{equation*} under the equivalence relation $(s, x) \sim (t,x)$ if $s\in |S|, t\in |T|, S\subseteq T$ and $s=t$. This construction comes equipped with continuous projection maps $\pi_1: \Delta U_\mathcal{U} \to |N\mathcal{U}|$ and  $\pi_2:\Delta U_\mathcal{U}\to U$ given by projecting onto the first and second factor, respectively. 
\lem{The fiber projecting map $\pi_2:\Delta U_\mathcal{U}\to U$ is a homotopy 
equivalence.} 
\begin{proof}[sketch.] As $U$ is assumed to be paracompact we can choose a
 partition of unity $\{\phi_i\}_{i\in I}$ subordinate to $\mathcal{U}$ and define 
 $g: U \to \Delta U_\mathcal{U}$ by 
 \begin{equation*}
   g(x) = \sum_{i\in I}\left(\phi_i(x)v_i, x\right), 
 \end{equation*}
where $v_i$ is the vertex corresponding to $U_i$. Then $\pi_2\circ g = \text{id}_U$ 
and it is not difficult to show that $g\circ \pi_2 \simeq \text{id}_{\Delta(\mathcal{U})}$. For a complete proof see~\cite{citeulike:1282830}.
\end{proof}

Now let $\mathcal{V} = \{V_j \setsuchthat j\in J\}$ be a finite cover of $V\supseteq U$ and $F: I\to J$ a map of covers. Recall that $|F|: |N\mathcal{U}|\to |N\mathcal{V}|$ denotes the continuous map defined on the vertices by the induced simplicial map between the nerves. If we let ${\rm inc}_U^V : U\hookrightarrow V$ denote the inclusion of $U$ into $V$ we get the commutative diagram
\begin{center}
\begin{tikzcd}
U \arrow{r}{\text{inc}_U^V}
&V\\
\Delta U_\mathcal{U}\arrow{r}{|F|\times {\rm inc}_U^V}\arrow{u}{\pi_2}[swap]{\simeq}\arrow{d}[swap]{\pi_1}
&\Delta V_\mathcal{V}\arrow{u}{\simeq}[swap]{\pi_2}\arrow{d}{\pi_1}\\
\, |N\mathcal{U}|\arrow{r}[swap]{|F|}
&\,|N\mathcal{V}|
\end{tikzcd}.
\end{center}
By passing to (singular) homology and using that $\pi_2$ is a homotopy equivalence we can reverse arrows to find
the following commutative diagram: 
\begin{equation}
\begin{tikzcd}
H_*(U) \arrow{r}{\left(\text{inc}_U^V\right)_*}\arrow{d}{}[swap]{(\pi_1)_*\circ (\pi_2)_*^{-1}}
&H_*(V)\arrow{d}{(\pi_1)_*\circ (\pi_2)_*^{-1}}\\
 H_*(|N\mathcal{U}|)\arrow{r}[swap]{|F|_*}
& H_*(|N\mathcal{V}|)
\end{tikzcd}
\label{eq:commutative}
\end{equation}

\begin{example} \label{ex:hocolim}
  Note that Diagram~\eqref{eq:commutative} does not commute on the level
  of spaces: let $\mathcal{U} = \{U\}$ and $\mathcal{V} = \{U, V\}$
  where $U\cap V\neq\emptyset$. If $x\in U\cap V$ then $(|F|\circ
  \pi_1\circ g)(x)$ is a point in $|N\mathcal{U}|$ whereas $(\pi_1\circ
  g \circ {\rm inc})(x)$ can be any point along the edge $|[U,
  V]|$, depending on the choice of partition of unity. See
  \figref{hocolim}.
\end{example}
\begin{figure}[htpb]
  \centering
  \quad\quad\quad\quad\begin{tikzpicture}[font=\small]
    \draw[right hook->] (0.7,0) -- (2.05, 0);
    \draw[right hook->] (0.7,-2) -- (1.5, -2);
    \draw[->] (0, -0.7) -- (0, -1.3);
    \draw[->] (3, -0.7) -- (3, -1.5);

    \begin{scope}[xshift=0cm, yshift=0cm]
      \draw[thick, postaction={white,fill}]{
        (0, 0) circle (0.5)
      };
      \coordinate (x) at (0.5,0);
      \filldraw[black] (x) circle (2pt);
      \node[anchor=south east] () at (-0.4, 0.4) {$U$};
    \end{scope}

    \begin{scope}[xshift=3cm, yshift=0cm]
      \draw[thick, postaction={white,fill}]{
        (-0.25, 0) circle (0.5)
        (0.25, 0) circle (0.5)
      };
      \coordinate (x) at (0.25,0);
      \filldraw[black] (x) circle (2pt);
      \node[anchor=south east] () at (-0.65, 0.4) {$U$};
      \node[anchor=south west] () at (0.65, 0.4) {$V$};
    \end{scope}

    \begin{scope}[xshift=0cm, yshift=-2cm]
      \draw[thick, postaction={white,fill}]{
        (0, 0) circle (0.5)
      };
      \coordinate (x) at (0.5,0);
      \filldraw[black] (x) circle (2pt);
      \node[anchor=north east] () at (-0.4, -0.4) {$\Delta U_{\{U\}}$};
    \end{scope}

    \begin{scope}[xshift=3cm, yshift=-2cm]
      \draw (-0.3, 0) arc (180:360:0.3 and 0.7);
      \draw (-0.6, 0) arc (180:360:0.6 and 1.0);

      \draw[thick, postaction={white,fill}]{
        (-0.8, 0) ellipse (0.5 and 0.3)
      };
      \draw[thick, postaction={white,fill}]{
        (0.8, 0) ellipse (0.5 and 0.3)
      };

      \draw[densely dotted] (-0.45, 0) ellipse (0.15 and 0.1);
      \draw[densely dotted] (0.45, 0) ellipse (0.15 and 0.1);

      \draw[densely dotted] (-0.3, 0) arc (180:360:0.3 and 0.7);
      \draw[densely dotted] (-0.6, 0) arc (180:360:0.6 and 1.0);

      \coordinate (x) at (-0.3,0);
      \filldraw[black] (x) circle (2pt);

      \coordinate (y) at (0,-0.85);
      \filldraw[black] (y) circle (2pt);
      \node[anchor=north west] () at (0.8, -0.4) {$\Delta (U\cup V)_{\{U,V\}}$};
    \end{scope}
  \end{tikzpicture}
  \caption{This diagram is an example of the diagram from \exref{hocolim} not commuting on the level of spaces.} \label{fig:hocolim}
\end{figure}
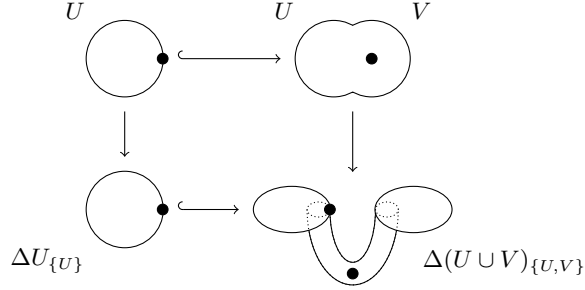

\begin{definition}
A cover is said to be \idf{good} if every finite intersection of its sets is contractible.
\end{definition}
\noindent The following theorem is one of the great pillars of computational 
algebraic topology. It allows us to use discrete information to capture the 
topology of a continuous space. For a proof see Section 4.G.\ of~\cite{citeulike:1282830}. 
\teo{If $\mathcal{U}$ is a good cover, then the base projection map $\pi_1: \Delta U_\mathcal{U} \to |N\mathcal{U}|$ is a homotopy equivalence.}
\cor{If $\mathcal{U}$ is a good cover, then the composition $(\pi_1)_*\circ (\pi_2)_*^{-1}$ is an isomorphism.
\label{cor:nerve}}

\subsection{A sandwich theorem for sequences of covers} \label{sec:sandwich}
We will use the results from the previous section to prove a sandwich type
theorem for sequences of covers. The idea is that if a sequence of covers can be 
sandwiched between two sequences of good covers, then the persistence module
associated to the middle sequence approximates the persistence modules
associated to the good covers. 

Let $\left(\mathcal{U}, F_\mathcal{U}\right), \left(\mathcal{V},F_{\mathcal{V}}\right)$ 
and $\left(\mathcal{W}, F_{\mathcal{W}}\right)$ be sequences of covers
satisfying
\begin{equation*}
U(\alpha)\subseteq V(\alpha)\subseteq W(\alpha) \subseteq U(c\alpha)
\end{equation*}
together with maps of covers
\begin{align*}
F_{\mathcal{V},\mathcal{W}}^{\alpha, \alpha^\prime}&: \mathcal{V}(\alpha) \to \mathcal{W}(\alpha^\prime) & 
F_{\mathcal{W},{\mathcal{V}}}^{\alpha, c\alpha^\prime}: \mathcal{W}(\alpha) \to \mathcal{V}(c\alpha^\prime)
\end{align*} 
for all $\alpha^\prime\geq\alpha$ and a fixed constant $c\geq 1$. Moreover, we assume that the maps
of covers satisfy the following coherence relations: 
\begin{align}
F_{\mathcal W}^{\alpha^\prime, \alpha^{\prime\prime}}\circ F_{{\mathcal V},{\mathcal W}}^{\alpha, \alpha^\prime} = 
F_{\mathcal V,\mathcal W}^{\alpha^\prime, \alpha^{\prime\prime}}\circ F_{\mathcal V}^{\alpha, \alpha^\prime}& ~ & F_{\mathcal W,\mathcal V}^{\alpha/c, c\alpha^\prime}
\circ F_{\mathcal V,\mathcal W}^{\alpha/c, \alpha/c} = F_{\mathcal V}^{\alpha/c, c\alpha^\prime}
\label{eq:coher}
\end{align}
for all $\alpha^{\prime\prime}\geq\alpha^\prime\geq\alpha$.

For notational simplicity we let $\eta_{\mathcal U,\mathcal V}^{\alpha,\alpha^\prime} = |F_{\mathcal U,\mathcal V}^{\alpha,\alpha^\prime}|_*$
and accordingly for the other maps of covers above. 
From Corollary \ref{cor:nerve} we know that if $\left(\mathcal{U}, F_{\mathcal{U}}\right)$ and $\left(\mathcal{W}, F_{\mathcal{W}}\right)$
are sequences of good covers, then there exist unique linear maps
$\eta_{\mathcal{U},{\mathcal{V}}}^{\alpha, \alpha^\prime}$, $\eta_{\mathcal{U},\mathcal{W}}^{\alpha,\alpha^\prime}$
and $\eta_{\mathcal{W},\mathcal{U}}^{\alpha, c\alpha^\prime}$ making the following diagrams
commute:
\begin{equation*}
\begin{tikzcd}
H_p(U(\alpha))\arrow{d}{(\pi_1)_*\circ(\pi_2)_*^{-1}}[swap]{\cong}\arrow{r}{\left({\rm inc}_{U(\alpha)}^{V(\alpha^\prime)}\right)_*} 
& H_p(V(\alpha^\prime))\arrow{d}[swap]{(\pi_1)_*\circ(\pi_2)_*^{-1}} & H_p(U(\alpha))\arrow{d}{(\pi_1)_*\circ(\pi_2)_*^{-1}}[swap]{\cong}\arrow{r}{\left({\rm inc}_{U(\alpha)}^{W(\alpha^\prime)}\right)_*} & H_p(W(\alpha^\prime))\arrow{d}{\cong}[swap]{(\pi_1)_*\circ(\pi_2)_*^{-1}} \\
H_p(|N\mathcal{U}(\alpha)|)\arrow{r}[swap]{\eta_{\mathcal{U},{\mathcal{V}}}^{\alpha, \alpha^\prime}} & H_p(|N\mathcal{V}(\alpha^\prime)|)
& H_p(|N\mathcal{U}(\alpha)|)\arrow{r}[swap]{\eta_{\mathcal{U},\mathcal{W}}^{\alpha,\alpha^\prime}} & H_p(|N\mathcal{W}(\alpha^\prime)|)
\end{tikzcd}
\end{equation*}
\begin{equation*}
\begin{tikzcd}
~&H_p(W(\alpha))\arrow{d}{(\pi_1)_*\circ(\pi_2)_*^{-1}}[swap]{\cong}\arrow{r}{\left({\rm inc}_{W(\alpha)}^{U(c\alpha^\prime)}\right)_*} & H_p(U(c\alpha^\prime))\arrow{d}{\cong}[swap]{(\pi_1)_*\circ(\pi_2)_*^{-1}} \\
& H_p(|N\mathcal{W}(\alpha)|)\arrow{r}[swap]{\eta_{\mathcal{W},\mathcal{U}}^{\alpha, c\alpha^\prime}} & H_p(|N\mathcal{U}(c\alpha^\prime)|)
\end{tikzcd}
\end{equation*}
\noindent Hence, there are well-defined linear maps 
\begin{align}
\phi_\alpha = \eta_{\mathcal{U},\mathcal{V}}^{\alpha, c\alpha}&: H_p(|N\mathcal{U}(\alpha)|)
\to H_p(|N\mathcal{V}(c\alpha)|) \notag
\\
\psi_\alpha = \eta_{\mathcal{W},\mathcal{U}}^{\alpha,c\alpha}\circ\eta_{\mathcal{V},\mathcal{W}}^{\alpha,\alpha}&: 
 H_p(|N\mathcal{V}(\alpha)|) \to H_p(|N\mathcal{U}(c\alpha)|)
 \label{eq:int2}
\end{align}
\noindent Also, note that the map $\eta_{\mathcal{W},\mathcal{V}}^{\alpha,c\alpha^\prime}$ is the unique map that
makes Diagram \eqref{eq:commutative} commute. 
\teo{If $\left(\mathcal{U}, F_{\mathcal{U}}\right)$ and $\left(\mathcal{W}, F_{\mathcal{W}}\right)$
are sequences of good covers, then the families of homomorphisms $\{\phi_\alpha\}_{\alpha\in [0, \infty)}$ and $\{\psi_\alpha\}_{\alpha\in[0, \infty)}$ 
defined in Equation \eqref{eq:int2} satisfy the diagrams of Definition 3.
 In particular, the persistence modules 
 \begin{equation*}
   \left(H_p(|N\mathcal{U}|), \eta_{\mathcal{U}}\right)  \text{ and } \left(H_p(|N\mathcal{V}|), \eta_{\mathcal{V}}\right)
 \end{equation*}
 are $c$-approximate.
\label{teo:sandwich}}

\begin{proof}
We need to show that the following four relations in Definition~\ref{def:capprox} are satisfied for all $\alpha\leq \alpha^\prime$:
\begin{align}
  &\psi_{\alpha^\prime}\circ \eta_{\mathcal{V}}^{\alpha, \alpha^\prime}\circ \phi_{\alpha/c} = \eta_{\mathcal{U}}^{\alpha/c, c\alpha^\prime}&  
  &\psi_{\alpha^\prime}\circ\eta_{\mathcal{V}}^{\alpha, \alpha^\prime} = \eta_{\mathcal{U}}^{c\alpha, c\alpha^\prime}\circ\psi_\alpha \nonumber \\
  &\phi_{\alpha^\prime}\circ\eta_{\mathcal{U}}^{\alpha, \alpha^\prime}\circ \psi_{\alpha/c} = \eta_{\mathcal{V}}^{\alpha/c, c\alpha^\prime}&
  &\phi_{\alpha^\prime}\circ\eta_{\mathcal{U}}^{\alpha, \alpha^\prime} = \eta_{\mathcal{V}}^{c\alpha, c\alpha^\prime}\circ\phi_\alpha
\label{eq:fourrel}
\end{align}
It follows from the uniqueness of the above linear maps, and the 
associativity of the maps in a sequence of covers, that any map composed out of the maps
\begin{equation}
\eta_\mathcal{U}^{-,-}, \eta_{\mathcal{V}}^{-,-}, \eta_\mathcal{W}^{-,-}, \eta_{\mathcal U, \mathcal W}^{-,-}, \eta_{\mathcal W, \mathcal U}^{-, -}
, \eta_{\mathcal W, \mathcal V}^{-, -} \hspace{0.2cm}{\rm and}\hspace{0.2cm} \eta_{\mathcal U, \mathcal V}^{-,-}\label{eq:specmap}\end{equation}
is uniquely defined by its domain and co-domain. That, together with the
coherence relations of Equation \eqref{eq:coher}, will prove the theorem. We will do the top left case of Equation~\eqref{eq:fourrel} in full detail
whereas we will refer to uniqueness arguments in the other three cases. 
\begin{description}
\item[Top left:]
\begin{align*}
& \psi_{\alpha^\prime}\circ \eta_{\mathcal{V}}^{\alpha, \alpha^\prime}\circ \phi_{\alpha/c} \\
&= \eta_{\mathcal{W},\mathcal{U}}^{\alpha^\prime,c\alpha^\prime}\circ\eta_{\mathcal{V},\mathcal{W}}^{\alpha^\prime,\alpha^\prime}\circ
\eta_\mathcal{V}^{\alpha, \alpha^\prime}\circ \eta_{\mathcal{U}, \mathcal{V}}^{\alpha/c, \alpha} \\
&= \eta_{\mathcal{W}, \mathcal{U}}^{\alpha^\prime, c\alpha^\prime}\circ \eta_{\mathcal{V}, \mathcal{W}}^{\alpha, \alpha^\prime}\circ \eta_{\mathcal{U}, \mathcal{V}}^{\alpha/c, \alpha} \hspace{1cm} \\ 
&= \eta_{\mathcal{W}, \mathcal{U}}^{\alpha^\prime, c\alpha^\prime}\circ \eta_{\mathcal{V}, \mathcal{W}}^{\alpha, \alpha^\prime}\circ (\pi_1\circ\pi_2^{-1})_*\circ \left({\rm inc}_{U(\alpha/c)}^{V(\alpha)}\right)_*\circ (\pi_1\circ\pi_2^{-1})_*^{-1}\\
&= \eta_{\mathcal{W}, \mathcal{U}}^{\alpha^\prime, c\alpha^\prime}\circ (\pi_1\circ\pi_2^{-1})_*\circ \left({\rm inc}_{V(\alpha)}^{W(\alpha^\prime)}\right)_* \left({\rm inc}_{U(\alpha/c)}^{V(\alpha)}\right)_*\circ (\pi_1\circ\pi_2^{-1})_*^{-1}\\
&= (\pi_1\circ\pi_2^{-1})_*\circ \left({\rm inc}_{U(\alpha/c)}^{U(c\alpha^\prime)}\right)_*\circ (\pi_1\circ \pi_2^{-1})_*^{-1} \\
& =\eta_{\mathcal{U}}^{\alpha/c, c\alpha^\prime}
\end{align*}
The second equality follows from the coherence relations. 
\item[Bottom left:] By definition, $\phi_{\alpha^\prime}\circ \eta_{\mathcal{U}}^{\alpha, \alpha^\prime}\circ \psi_{\alpha/c} 
= \eta_{\mathcal{U}, \mathcal{V}}^{\alpha^\prime, c\alpha^\prime}\circ \eta_{\mathcal{U}}^{\alpha, \alpha^\prime}\circ 
\eta_{\mathcal{W},\mathcal{U}}^{\alpha/c, \alpha}\circ\eta_{\mathcal{V},\mathcal{W}}^{\alpha/c, \alpha/c}$. Using that the
composition of the three leftmost maps has same domain and co-domain as $\eta_{\mathcal{W},\mathcal{V}}^{\alpha/c, c\alpha^\prime}$
, we are left with $\eta_{\mathcal{W}, \mathcal{V}}^{\alpha/c,c\alpha^\prime}\circ \eta_{\mathcal{V},\mathcal{W}}^{\alpha/c,\alpha/c}=
\eta_{\mathcal{V}}^{\alpha/c, c\alpha^\prime}$. Here the last equality
follows from Equation \eqref{eq:coher}. 
\item[Top right:] From the coherence relations in Equation \eqref{eq:coher} we find $\psi_{\alpha^\prime}\circ\eta_{\mathcal{V}}^{\alpha, \alpha^\prime} 
= \eta_{\mathcal{W},\mathcal{U}}^{\alpha^\prime, c\alpha^\prime}\circ\eta_{\mathcal{V},\mathcal{W}}^{\alpha^\prime,\alpha^\prime}\circ \eta_{\mathcal{V}}^{\alpha, \alpha^\prime}
= \eta_{\mathcal{W},\mathcal{U}}^{\alpha^\prime, c\alpha^\prime}\circ \eta_{\mathcal{W}}^{\alpha, \alpha^\prime}\circ \eta_{\mathcal{V},\mathcal{W}}^{\alpha, \alpha}$.
Lastly, we note that 
$\eta_{\mathcal{W},\mathcal{U}}^{\alpha^\prime,c\alpha^\prime}\circ \eta_{\mathcal{W}}^{\alpha, \alpha^\prime}$ 
and $\eta_{\mathcal{U}}^{c\alpha, c\alpha^\prime}\circ \eta_{\mathcal{W},\mathcal{U}}^{\alpha, c\alpha}$ 
are equal by uniqueness. 
\item[Bottom right:] Both sides of the equation are composed out of maps from \eqref{eq:specmap}. 
\end{description}
\end{proof}
The following is a corollary of the proof.
\begin{corollary}
Any two of the persistence modules
\begin{equation*}
 \left(H_p(|N\mathcal{U}|), 
 \eta_{\mathcal{U}}\right), \left(H_p(|N\mathcal{V}|), \eta_{\mathcal{V}}\right) \text{ and }
\left(H_p(|N\mathcal{W}|), \eta_{\mathcal{W}}\right)
\end{equation*}
 are $c$-approximate. 
\end{corollary}

Note that we do not require the covers in the sequence
$(\mathcal{V},F_{\mathcal{V}})$ to be good. One application of this,
which will be pursued in the next section, is the following. Let $(\mathcal{U}, F_{\mathcal{U}})$
be a sequence of good covers and $(\mathcal{V}, F_\mathcal{V})$ another sequence of covers
where each open set in $\mathcal{V}(\alpha)$ is a union of open sets in $\mathcal{U}(\alpha)$. Thus, we have a 
map of covers $\mathcal{U}(\alpha) \to \mathcal{V}(\alpha^\prime)$, and more interestingly, a linear map
$H_p(|N\mathcal{U}(\alpha)|)\to H_p(|N\mathcal{V}(\alpha^\prime)|)$. However, we do not have a map of covers the other way around, 
so it is a priori not clear how to define the interleaving map in the opposite direction. This is
illustrated in Figure \ref{fig:nonexist}.
The previous theorem tells us that such a map can be constructed and gives an upper bound on the
bottleneck distance between the associated persistence modules.

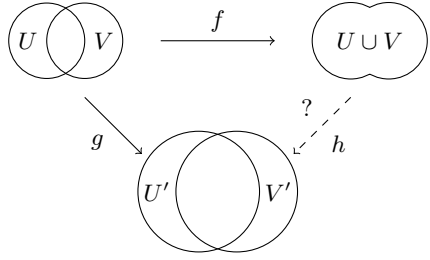
\begin{figure}[htpb]
  \centering
  \begin{tikzpicture}[font=\small]
    \begin{scope}[xshift=0cm, yshift=0cm]
      \draw (-0.25,0) circle (0.5);
      \draw (0.25,0) circle (0.5);
      \node () at (-0.5,0) {$U$};
      \node () at (0.5,0) {$V$};
    \end{scope}
    \begin{scope}[xshift=4cm, yshift=0cm]
      \draw[thick, postaction={white,fill}]{
        (-0.25, 0) circle (0.5)
        (0.25, 0) circle (0.5)
      };
      \node () at (0,0) {$U\cup V$};
    \end{scope}
    \begin{scope}[xshift=2cm, yshift=-2cm]
      \draw (-0.25,0) circle (0.8);
      \draw (0.25,0) circle (0.8);
      \node () at (-0.8,0) {$U'$};
      \node () at (0.8,0) {$V'$};
    \end{scope}
    \draw[->] (1.25,0) -- (2.75,0) node[auto,midway] {$f$};
    \draw[->] (0.25,-0.75) -- (1, -1.5) node[auto,midway,swap] {$g$};
    \draw[dashed, ->] (3.75,-0.75) -- (3, -1.5) node[auto,midway] {$h$} node[auto,midway,swap] {?};
  \end{tikzpicture}
  \caption{The map of covers $f$ is defined as sending a ball to the
    union it belongs to, and $g$ as the obvious map of covers arising
    from $U\subseteq U'$ and $V\subseteq V'$. There is no map of covers $h$ making
    the diagram commute on the level of covers. } \label{fig:nonexist}
\end{figure}

\section{Approximating the \Cech{} complex in Euclidean space} \label{sec:cechrips}
In this section we construct two different approximation schemes for the 
\Cech{} persistence module built on a finite set of points in Euclidean space. 

It is clear that for any $c$-approximation of the \Cech{} persistence
module, a $\sqrt{2}c$-approximation can be had via the Vietoris--Rips
complex built on its $1$-skeleton. For a treatment of approximate
Vietoris--Rips complexes in general metric spaces see~\cite{dey:2012:simplicialmaps, Sheehy13}.

\subsection{Linear-size approximation of the \Cech{} persistence module} \label{sec:linsize}
This section is an adaption of the work in~\cite{dey:2012:simplicialmaps} to the case of \Cech{}
complexes in Euclidean space. Throughout this section, $P\subseteq\mathbb{R}^n$.
\df{For a set of points $P$, we say that $P^\prime\subseteq P$ is a \idf{$\delta$-net} of $P$ if 
\begin{enumerate}
\item for every $p\in P$ there exists a $p^\prime \in P'$ such that $||p-p^\prime|| \leq \delta$
\item for any $p,q\in P'$, $||p-q||>\delta$.
\end{enumerate}
}
Choose parameters $\alpha_0,\epsilon\geq 0$ and define a
sequence of point sets $P_k$ for $k=0,1, \dotsc, m$ such that $P_0 =
P$ and $P_{k+1}$ is an $\alpha_0\epsilon^2(1+\epsilon)^{k-1}$-net of
$P_k$. We refer to such a collection $P_0,\dotsc,P_m$ as a
\idf{net-tree}. Furthermore, let $\mathcal{C}(P_k;\alpha)$ be the
\Cech{} complex at scale $\alpha$ built upon the vertex set $P_k$, and
$U(P_k; \alpha)$ the union of open balls of radius $\alpha$ centered
at each point in $P_k$. We clearly have maps $\pi_k: P_k\to P_{k+1}$
which send a vertex $p\in P_k$ to its most nearby vertex in $P_{k+1}$.

\lem{For every $k=0, \ldots, m-1$ we have inclusions 
  \begin{equation*}
    U(P; \alpha_0(1+\epsilon)^k) \subseteq U(P_{k+1}; \alpha_0(1+\epsilon)^{k+1}).  
  \end{equation*}\label{lem:nettree-inclusion}}
\begin{proof}
Let $p\in P=P_0$ and $x\in\mathbb{R}^n$ be any point such that $||p -x||<\alpha_0(1+\epsilon)^k$. Since $P_{1}$ is an $\alpha_0\epsilon^2(1+\epsilon)^{-1}$-net of $P$
we can find $\pi_0(p)\in P_1$ such that $||\pi_0(p) - p|| \leq \alpha_0\epsilon^2(1+\epsilon)^{-1}$. Similarly, we can find $p^\prime = (\pi_k\circ\cdots\circ\pi_0) (p)\in P_{k+1}$ such that
\begin{align*}
||p^\prime - x|| &=  ||\pi_k\circ\cdots\circ\pi_0(p) - x||\\
&\leq ||p - x || + \sum_{i=0}^k \alpha_0\epsilon^2(1+\epsilon)^{i-1}\\
&\leq ||p-x|| + \frac{\alpha_0\epsilon^2}{1+\epsilon}\cdot \frac{(1+\epsilon)^{k+1}-1}{\epsilon}\\
&< \alpha_0(1+\epsilon)^k + \alpha_0\epsilon(1+\epsilon)^{k} = \alpha_0(1+\epsilon)^{k+1}.
\end{align*}
\end{proof}
\noindent 
In particular, for $p\in P_k$ we have that $B(p; \alpha_0(1+\epsilon)^k) 
\subseteq B(\pi_k(p); \alpha_0(1+\epsilon)^{k+1})$, and thus 
$\pi_k: P_k \to P_{k+1}$ is a map
of covers 
\begin{equation*}
\pi_k: \mathcal{U}(P_k; \alpha_0(1+\epsilon)^k) \to \mathcal{U}(P_{k+1}; 
\alpha_0(1+\epsilon)^{k+1}).
\end{equation*}
 Using this we define a sequence of covers
associated to the net tree by defining $$\mathcal{U}^{\rm net}(P;\alpha) 
= \mathcal{U}(P_k; \alpha_0(1+\epsilon)^k)$$ where $k$ is the greatest 
integer such that $\alpha_0(1+\epsilon)^k \leq \alpha$. The maps between 
the covers are given by compositions of $\pi_k$'s. We will denote the induced
sequence of nerves by $\mathcal{C}^{\rm net}(P)$ and the associated persistence module
by $(H_p(\mathcal{C}^{\rm net}(P)), \pi_*)$. Recall that with this notation we have
that $$U^{\rm net}(P; \alpha) = U(P_k; \alpha_0(1+\epsilon)^k) = \bigcup_{p\in P_k} B(p; \alpha_0(1+\epsilon)^k)$$

\prop{The persistence modules $(H_p(\mathcal{C}^{\rm net}(P)), \pi_*)$ and $(H_p(\mathcal{C}(P)), {\rm id}_*)$ are  $(1+\epsilon)^2$-approximate.}

\begin{proof}
Using that $U^{\rm net}(P; \alpha) = U(P_k; \alpha_0(1+\epsilon)^k)$ together with Lemma \ref{lem:nettree-inclusion} we have the chain of inclusions
\begin{align*}
  U^{\rm net}(P;\alpha) \subseteq U(P; \alpha) \subseteq U(P; \alpha_0(1+\epsilon)^{k+1}) &\subseteq U(P_{k+2}; \alpha_0(1+\epsilon)^{k+2}) \\
  &= U^{\rm net}(P;\alpha(1+\epsilon)^2).
\end{align*}
The rest of the proof follows by applying Theorem \ref{teo:sandwich} 
with $\mathcal{U} = \mathcal{U}^{\rm net}(P)$ and $\mathcal{V} = \mathcal{W} 
= \mathcal{U}(P)$. 

\end{proof}

\prop{Let $P\subseteq\mathbb{R}^n$ be a set of $m$ points. Then the number 
of $p$-simplices in $\mathcal{C}^{\rm net}(P;\alpha_0(1+\epsilon)^k)$ is 
$\mathcal{O}\left((\frac{1}{\epsilon})^{\mathcal{O}(np)}m\right)$.}
\begin{proof}
  This is Theorem 6.3 in \cite{dey:2012:simplicialmaps} together with
  the fact that the doubling dimension of $\mathbb{R}^n$ is
  $\mathcal{O}(n)$
\end{proof}

The net-tree construction exhibits great theoretical properties both with regards
to approximating the \Cech{} persistence module and in terms
of size complexity. In practice however, as we shall see in Section
\ref{sec:experiments}, the complex often grows too large to be
stored. Not doing a single collapse between scale
$\alpha_0(1+\epsilon)^k$ and scale $\alpha_0(1+\epsilon)^{k+1}$ will
in many situations introduce too many new simplices. To mend this we
introduce a complex which allows for more numerous collapses, at the
expense of computation time and poorer error bounds.
\label{sec:nettree}

\subsection{Approximations through non-good covers}
We propose a general framework to approximate persistence modules associated to
sequences of good covers. Using this framework we give an explicit approximation
of the \Cech{} persistence module in Euclidean space. 

Let $(\mathcal{U},F)$ be a sequence of covers with index sets $\{I(\alpha)\}_{\alpha\geq 0}$ and
$J(I(\alpha))$ a partition of $I(\alpha)$. We make the following assumption on the 
partitions: if $J\in J(I(\alpha))$ then for all $\alpha^\prime\geq\alpha$ there exists
$J^\prime\in J(I(\alpha^\prime))$ such that $J\subseteq J^\prime$. In other words, if two
elements are partitioned together at some scale $\alpha$, they will be partitioned together
at all scales $\alpha^\prime\geq\alpha$. Moreover, if $J\in J(I(\alpha))$
then $F^{\alpha,\alpha^\prime}(J)$ denotes the set $J^\prime\in J(I(\alpha^\prime))$ such that
$J\subseteq J^\prime$. 

\lem{For each $\alpha\geq 0$, let $J(I(\alpha))$ be a partition of
 $I(\alpha)$ as described above.
 Then the pair $(\widetilde{\mathcal{U}}, F)$ with
 \begin{equation*}
   \widetilde {\mathcal U}(\alpha) = \left\{\widetilde{U}_J(\alpha) = \bigcup_{j \in J} U_j(\alpha) \setsuchthat J\in J(I(\alpha))\right\}
 \end{equation*}
 is a sequence of covers.}
\begin{proof}
This follows from that $J\subseteq F^{\alpha,\alpha^\prime}(J)$ for all $J\in J(I(\alpha))$.  
\end{proof}
For such a choice of partitions we say that $(\widetilde{\mathcal{U}}, F)$ is \idf{coarsening} of $(\mathcal{U},F)$. 

Let $(\widetilde{\mathcal{U}}(P), {\rm id})$ be any coarsening of the \Cech{} sequence of covers $\mathcal{U}(P)$ on a finite point set $P\subseteq\mathbb{R}^n$. Furthermore, define an associated sequence of good covers $(\Conv\widetilde{\mathcal{U}}, {\rm id})$ where 
\begin{equation*}
  \Conv{\widetilde{\mathcal{U}}}(\alpha) = \left\{\Conv\left(\widetilde{U}_k(\alpha)\right) \setsuchthat \widetilde{U}_k(\alpha)\in \widetilde {\mathcal U}(\alpha) \right\},
\end{equation*}
and $\Conv(-)$ denotes the convex hull. In the following proposition
 $(\widetilde{\mathcal{C}}(P),\id)$ denotes the induced sequence of nerves of 
 $(\widetilde{\mathcal{U}}(P),\id)$. 

\prop{If there exists a constant $c\geq 1$ such that $\Conv\left(\widetilde{U}_J
(\alpha)\right) \subseteq \bigcup_{j\in J} U_j(c\alpha)$ for all $\alpha\geq 0$ and 
all $J\in J(\alpha)$, then the persistence modules $\left(H_p(\mathcal{C}(P)), 
{\rm id}_*\right)$ and \par \noindent $\left(H_p(\widetilde{\mathcal{C}}(P)), {\rm id}_*\right)$ are 
$c$-approximate.
\label{prop:approx}}
\begin{proof}
We will use Theorem \ref{teo:sandwich}. We see that the inclusion condition is satisfied
by assumption: 
\begin{equation*}
U(\alpha) \subseteq U(\alpha) \subseteq \bigcup_{J\in J(\alpha)}\Conv\left(\widetilde{U}_J
(\alpha)\right) \subseteq U(c\alpha).
\end{equation*}
Moreover, $\widetilde{\mathcal{U}}(P; \alpha)$ and $\Conv\widetilde{\mathcal{U}}
(P;\alpha)$ have the same indexing set, so the coherence relations of Equation 
\eqref{eq:coher} are trivially satisfied.
\label{prop:coarsening}
\end{proof}
We see that every time we make our cover coarser, the number of $0$-simplices in the nerve
is reduced, and hence so is the size of the simplicial complex. 
\label{sec:nongood}
\subsubsection{An explicit approximation} \label{sec:explicitapprox}
In the previous section we provided a general framework for constructing $c$-approximations
to the \Cech{} persistence module. We now give an explicit construction using 
Proposition~\ref{prop:approx}. 
\lem{Let $P=\{p_0,p_1, \ldots, p_k\}\subset\mathbb{R}^n$ where $p_0 = 0$ and $||p_i|| \leq \alpha$ for all $i$. 
Then for any point $x\in\Conv(P)$ there exists $p_i\in P$ such that $||x-p_i||\leq \alpha/\sqrt{2}$.
\label{lem:consphere}}
\begin{proof}
By definition of $p_0$ we may assume without loss of generality that $x = (x_1, 0, \ldots, 0)$ where
$x_1 > \alpha/\sqrt{2}$. Let $p_i = (p_{i,1}, p_{i,2}, \ldots p_{i,n})$
be a point in $P$ such that $p_{i,1} \geq p_{j,1}$ for every other $j$, and assume that $||x-p_i|| > \alpha/\sqrt{2}$. Using the law of cosines:
\begin{align*}
\alpha^2 \geq ||p_i||^2 &= ||(p_i-x) + x||^2 = ||p_i-x||^2 + ||x||^2 - 2||p_i-x||\cdot||x||\cos(\angle p_0xp_i)\\
&> \frac{\alpha^2}{2} + \frac{\alpha^2}{2} - 2||p_i-x||\cdot ||x||\cos(\angle p_0xp_i)\\
&= \alpha^2 - 2||p_i-x||\cdot ||x||\cos(\angle p_0xp_i)
\end{align*} 
implying that $\cos(\angle p_0xp_i)>0$. By application of the Euclidean inner product we find
\begin{equation*}
  (p_i -x)\cdot (-x) = -p_{i,1}\cdot x_1 + x_1^2 = ||p_i-x||\cdot ||x||\cdot\cos(\angle p_0xp_i) > 0
\end{equation*}
and therefore $p_{i,1} < x_1$, contradicting that $x$ was enclosed in the convex hull of $P$. 

\end{proof}
Figure~\ref{fig:consphereextreme} shows an extreme case of the previous Lemma.
\begin{figure}[htpb]
  \centering
  \begin{tikzpicture}[font=\small]
    \coordinate (p0) at (0,0);
    \coordinate (p1) at (225:2cm);
    \coordinate (p2) at (315:2cm);
    \draw (p0) -- (p1) -- (p2) -- cycle;
    \draw plot[mark=*] coordinates{(p0)} node[above] {$p_0$};
    \draw plot[mark=*] coordinates{(p1)} node[left] {$p_1$};;
    \draw plot[mark=*] coordinates{(p2)} node[right] {$p_2$};;
    \draw plot[mark=x] coordinates{(p0 |- p1)} node [below] {$x$};
  \end{tikzpicture}
  \caption{The vertices $p_0, p_1, p_2$ of an isosceles triangle $T=\Conv\{p_0,p_1,p_2\}$ with legs of length $\alpha$ and base of length $\sqrt{2}\alpha$ form an extreme case of Lemma~\ref{lem:consphere} as $x\in T$ lies a distance $\alpha/\sqrt{2}$ from every vertex.} \label{fig:consphereextreme}
\end{figure}
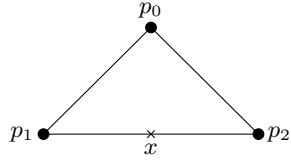

\prop{Let $\alpha\geq 0$ and $\epsilon \geq 0$. If $P=\{p_0, p_1, \dotsc, p_k\}\subset\mathbb{R}^n$ is a set of points such that
$||p_i-p_j|| \leq \epsilon\alpha$, then the following relation holds:
\begin{equation*}
  \Conv\left(\bigcup_{0\leq i \leq k}B(p_i; \alpha) \right) \subseteq \bigcup_{0\leq i \leq k}B\left(p_i; \alpha\sqrt{1+\frac{\epsilon^2}{2}}\right).
\end{equation*}
\label{prop:chull}
}
\begin{proof}
First, observe that we have the equality
\begin{equation*}
  \Conv\left(\bigcup_{0\leq i \leq k}B(p_i; \alpha) \right) = \left\{x\in\mathbb{R}^n \setsuchthat \exists y\in\Conv (P), ||x- y|| < \alpha\right\}.
\end{equation*}
Any point $x\in \Conv(P)$ is contained in the union $\cup_{0\leq i\leq k} B(p_i; \epsilon\alpha/\sqrt{2})$ by Lemma \ref{lem:consphere}. 
Thus, what remains to be shown is that the proposition holds true for any $x\in\mathbb{R}^n$ for which there is a $p\in \Conv\{p_{i_0}, \ldots, p_{i_k}\}$, $k\leq n-1$, such that $||x-p||<\alpha$.
The last inequality follows since $x$ is in the exterior
of the convex hull and the most nearby point cannot be strictly inside an $n$-simplex.  

Denote by $x^\prime$ the orthogonal projection of $x$ down on the affine space spanned by $\{p_{i_0}, \ldots, p_{i_k}\}$. If $x^\prime \in \Conv\{p_{i_0}, \ldots, p_{i_k}\}$ it follows from Lemma~\ref{lem:consphere} that there exists a $p_{i_j}$ such that 
\begin{equation*}
  ||p_{i_j} - x||^2 = ||p_{i_j} - x^\prime||^2 + ||x-x^\prime||^2 \leq \frac{\epsilon^2\alpha^2}{2} + \alpha^2 = \alpha^2\left(1+\frac{\epsilon^2}{2}\right).
\end{equation*}
If $x^\prime \not\in \Conv\{p_{i_0}, \dotsc, p_{i_k}\}$ it implies the existence of a point $p^\prime$ on the 
boundary of $\Conv\{p_{i_0}, \dotsc, p_{i_k}\}$ such that $||x-p^\prime|| \leq ||x-p|| < \alpha$ and we can repeat the process
for that point. This completes the proof as the proposition is trivially true if $k=0$. 

\end{proof}
\begin{figure}[htpb]
  \centering
  \begin{tikzpicture}[font=\small]
    \begin{scope}
      \coordinate (p0) at (0,0);
      \coordinate (p1) at (240:1cm);
      \coordinate (p2) at (300:1cm);
      \draw[black] \convexpath{p0,p2,p1}{0.52cm};
      \draw[thick, dashed, postaction={white, fill}]{
        (p0) circle (0.52)
        (p1) circle (0.52)
        (p2) circle (0.52)
      };
    \end{scope}
    \begin{scope}[xshift=5cm]
      \coordinate (p0) at (0,0);
      \coordinate (p1) at (240:1cm);
      \coordinate (p2) at (300:1cm);
      \draw[thick, dashed, postaction={white, fill}]{
        (p0) circle (0.8)
        (p1) circle (0.8)
        (p2) circle (0.8)
      };
      \draw[black] \convexpath{p0,p2,p1}{0.55cm};
    \end{scope}    
  \end{tikzpicture}
  \caption{Left: the convex hull of a union of three balls. Right: By increasing
  the radii of the balls their union eventually covers the convex hull.} \label{fig:chullinc}
\end{figure}
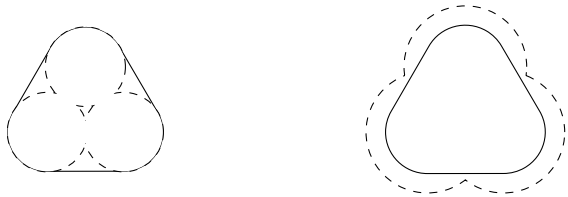
\noindent The previous proposition is illustrated in Figure \ref{fig:chullinc}.
 By combining Propositions \ref{prop:approx} and \ref{prop:chull} we 
have shown the following.
\prop{Let $\epsilon \geq 0$. Suppose ${\widetilde{\mathcal{U}}}(P)$ is a coarsening of $\mathcal{U}(P)$ 
with the property that for every $\alpha\geq 0$ and every pair of indices $i, j\in J\in J(I(\alpha))$, the inequality
$||p_i-p_j||\leq \alpha\cdot \epsilon$ holds. Then
$H_p(\widetilde{\mathcal{C}}(P), {\rm id}_*)$ is a $\sqrt{1+\epsilon^2/2}$-approximation
 of the \Cech{} persistence module built on $P$. }

The previous proposition allows us to build good approximations to the 
\Cech{} persistence module with far fewer simplices. A problem with this approach 
is that such a memory efficient construction comes at the expense of computing 
weights of simplices. As an example, if $J(I(\alpha))$ consists of $k$ partitions,
each with $m$ elements, then computing the smallest $\alpha$ at which they have a $k$-intersection has time complexity $\mathcal{O}(m^k)$. To mend this we seek methods to approximate
this persistence module by ones that are less computationally expensive. The next section details one method for doing so. 
\label{sec:approx}
\subsubsection{Choosing a representative}
Let $(\widetilde{\mathcal{U}}(P), {\rm id})$ be a coarsening of the \Cech{} sequence of covers
and for every $\alpha\geq 0$ and every $J\in J(\alpha)$ choose a representative $p_j\in P$, where $j\in J$.
Denote the set of representatives at scale $\alpha$ by $P_\alpha$. For every
$\alpha\geq 0$ we define the subcomplex $\mathcal{C}^{\rm rep}(P; \alpha)\subseteq \widetilde{\mathcal{C}}(P; \alpha)$
to be the smallest simplicial complex such that:
\begin{enumerate}
\item $\mathcal{C}(P_\alpha; \alpha)\subseteq \mathcal{C}^{\rm rep}(P; \alpha)$
\item ${\rm id}^{\alpha, \alpha^\prime}: \widetilde{\mathcal{C}}(P;\alpha)\to\widetilde{\mathcal{C}}(P;\alpha^\prime)$ restricts to a simplicial map $\mathcal{C}^{\rm rep}(P; \alpha) \to \mathcal{C}^{\rm rep}(P; \alpha^\prime)$
\end{enumerate}

The idea is to choose a set
of representatives, one for each element $J(I(\alpha))$, and use those representatives to approximate
the persistent homology computation. However, to get a well-defined sequence of simplicial complexes
and simplicial maps, we need to make sure that the image of a simplex spanned by one set of representatives
 is a simplex at a later filtration time, where the set of representatives may be different. Thus, our approximate
 complex contains the simplicial complex built on the set of representatives and, in addition, the images of simplices spanned
 by representatives at earlier filtration times. 

\prop{The persistence modules $(H_p(\widetilde{\mathcal{C}}(P)), \id_*)$ and
$(H_p(\mathcal{C}^{\rm rep}(P)), \id_*)$ are $\frac{1}{1-\epsilon}$-approximate.}
\begin{proof}
The simplicial complexes $\mathcal{C}^{\rm rep}(P_\alpha; \alpha)$ and
$\widetilde{\mathcal{C}}(P; \alpha)$ are defined over the same indexing set $J(I(\alpha))$
for every $\alpha\geq 0$. This follows from having chosen one representative for each
covering set of $\widetilde{\mathcal{U}}(\alpha)$. Now choose 
$x\in U_J\in \widetilde{\mathcal{U}}(\alpha)$,
where $||x-p_j||<\alpha$ for some $j\in J$, and let $p$ be the representative of
$\id^{\alpha,\alpha/(1-\epsilon)}(J)\in J(I(\alpha/(1-\epsilon)))$. Then 
\begin{equation*}
  ||p-x|| \leq ||p-p_j|| + ||x-p_j|| < \frac{\alpha\epsilon}{1-\epsilon} + \alpha = \frac{\alpha}{1-\epsilon}.
\end{equation*}
Hence, we have a map of covers $\widetilde{\mathcal{U}}(P;\alpha)\to \mathcal{U}(P_{\alpha/(1-\epsilon)}; \frac{\alpha}{1-\epsilon})$
which induces the first map of the composition 
\begin{equation*}
\widetilde{\mathcal{C}}(P; \alpha) \to \mathcal{C}\left(P_{\alpha/(1-\epsilon)}; \frac{\alpha}{1-\epsilon}\right)\subseteq \mathcal{C}^{\rm rep}\left(P_{\alpha/(1-\epsilon)}; \frac{\alpha}{1-\epsilon}\right)
\subseteq \widetilde{\mathcal{C}}\left(P; \frac{\alpha}{1-\epsilon}\right)
\end{equation*} The proof 
follows from application of Definition \ref{def:capprox} and Theorem \ref{teo:bottdistance}.
\end{proof}

\subsection{Relationship to graph induced complexes}
We conclude this section by briefly discussing a related construction introduced in~\cite{dey:gic} by Dey \textit{et al.}
\df{Let $G(V)$ be a graph with vertex set $V$ and let $\nu: V\to V^\prime$ be a vertex
map where $\nu(V) = V^\prime \subseteq V$. The \idf{graph induced complex} $\mathcal{G}(V, V^\prime, \nu)$ 
is defined as the simplicial complex where a $k$-simplex $[v_1^\prime, \ldots , v_{k+1}^\prime]$ is
in $\mathcal{G}(V, V^\prime, \nu)$ if and only if there exists a $(k+1)$-clique 
$\{v_1, \ldots, v_{k+1}\}\subseteq V$ such that $\nu(v_i) = v_i^\prime$ for each
$i\in\{1, \ldots, k+1\}$.}

First we note that a coarsening of a cover as defined at the beginning of Section~\ref{sec:nongood} induces a graph induced
complex. Indeed, just choose a representative for each partition and let $\nu$ be the map
which takes a vertex to its representative. This, together with a net-tree construction
as in Section \ref{sec:nettree}, is utilized in \cite{dey:2012:simplicialmaps} to construct a linear-size
approximation to the Vietoris--Rips persistence module. Constructing the analogue \Cech{} approximation is straightforward 
and it can be shown that it enjoys error bounds similar to what we proved in Section~\ref{sec:nettree}.
In fact, the analogue \Cech{} construction is nothing more than forming a coarsening of the \Cech{} sequence of covers 
where the process of partitioning covering sets is determined by a net-tree. 
Unfortunately, as discussed at the end of Section~\ref{sec:approx},
computing the $k$-intersections needed for this construction is very time consuming. 

\section{Computational experiments} \label{sec:experiments} 
This section details our implementation of the approximation schemes
described above, as well as some computational examples examining their
efficacy and practical applicability.

\subsection{Implementation}
We realize an implementation of the approximation schemes detailed in
Section~\ref{sec:nongood} as a C++ program in the following way.

The program takes as parameters $\epsilon\geq 0$ (as in
\secref{explicitapprox}), a maximal scale $\alphamax>0$ (as usual when
computing persistence), a maximal simplex dimension $D>0$ (as usual)
and $L\in\mathbb{N}$ (to be explained later). Given an input point
cloud $P=\{p_1,\dotsc,p_N\}\subseteq\mathbb{R}^d$, we first use
Müllner's \textit{fastcluster}~\cite{fastcluster} to compute its
hierarchical clustering $\HC(P)$ with the \emph{complete} linkage
criterion. This is considered a preprocessing step.

A \idf{cluster} is a pair $(p, X)$ with $p\in X\subseteq P$,
wherein $p$ will be called the cluster's
\idf{representative} and $X$ its \idf{members}. At initialization
time, we begin with $N$ clusters
\begin{equation*}
  c^0_1=(p_1,\{p_1\}), c^0_2=(p_2,\{p_2\}),\dotsc,c^0_N=(p_N, \{p_N\}).
\end{equation*}
and denote their enumeration by $C^0=\{1,\dotsc,N\}$.

We shall regard $\HC(P)$ as the data of a series of \idf{linkage
  events} of the form $(s, i,
j)\in\mathbb{R}\times\mathbb{N}\times\mathbb{N}$ ordered by the first
component, and (arbitrarily) with the convention that $i<j$. An event
like this signifies the linking of clusters $c^l_i=(p^l_i, X^l_i)$ and
$c^l_j=(p^l_j, X^l_j)$ at scale $s$, from which we form a new cluster
$c^{l+1}_i=(p^{l+1}_i, X^l_i\cup X^l_j)$ where $p^{l+1}_i\in X^l_i\cup
X^l_j$; in principle the new representative $p^{l+1}_i$ can be chosen
arbitrarily from $X^l_i\cup X^l_j$, but for heuristic reasons we pick
the point in the member set $X^l_i\cup X^l_j$ closest to that set's
centroid.

We maintain a priority queue $Q$ of simplices prioritized by their
persistence time. At initialization, the queue contains the
$0$-simplices $[1],\dotsc,[N]$ all at persistence time $0$. A simplex
tree, along with associated annotations and other data structures as
described in \secref{computing}, are also initialized empty. These
data structures that track homology will jointly be referred to as
$\PH$ below, and we shall abuse language and speak of a simplex as
``belonging to $\PH$'' when the simplex is present in the simplicial
complex. We also initialize $\alpha'=0$ and $l=0$ to begin with.

The implementation code then proceeds in the following steps:
\begin{enumerate}
\item If $Q$ is empty, we are done and go to step~\ref{1c1792}. If
  not, pop a simplex $\sigma$ and its persistence scale $\alpha$ from
  the front, and continue. \label{96ad1e}
\item If $\alpha > \alphamax$, we are done and go to
  step~\ref{1c1792}. Otherwise continue.
\item If $\sigma$ is not already in $\PH$, add it according to
  \secref{computing}. In both cases, continue.
\item If $\dim\sigma>D$, go to
  step~\ref{6a2693}. Otherwise, for each simplex
  $\tau\in\{\sigma\cup\{i\} \setsuchthat i\in C^l\}$:
  compute\footnote{Our implementation uses Gärtner's
    \textit{Miniball}~\cite{miniball} for this computation.} the
  radius $r_\tau$ of the smallest enclosing ball of the set
  $\{p^l_i\setsuchthat i\in\tau\}\subseteq P$, and add $\tau$ to $Q$
  at persistence scale $r_\tau$. Go to step~\ref{6a2693}.
\item If at least $L$ simplices have been added to $\PH$ since the
  last time this step was reached, we (possibly) perform a
  simplification by going to step~\ref{455d4d}. Otherwise go to
  step~\ref{96ad1e}. \label{6a2693}
  \begin{enumerate}
  \item For each linkage event $(s,i,j)\in\HC(P)$ for which
    $s\in[\alpha',\epsilon\alpha)$, perform the edge contraction
    $[i,j]\mapsto [i]$ according to \secref{computing}, taking care to
    adjust persistence times to reflect a (possible) series of
    inclusions to satisfy the link condition. If there were no linkage
    events in the given interval, go to step~\ref{96ad1e}. Otherwise,
    denote the clusters present after handling the linkage event, as
    explained earlier in this section, by
    \begin{equation*}
      \{C^{l+1}_{i_1},\dotsc,C^{l+1}_{i_{N_{l+1}}}\} \subseteq \{C^{l}_{j_1},\dotsc,C^{l}_{j_{N_{l}}} \}
    \end{equation*}
    and go to  step~\ref{ca688f}. \label{455d4d}
  \item Clear $Q$ and reset it to contain the $0$-simplices
    $[i_1],\dotsc,[i_{N_l}]$, all at persistence scale $0$. Update $l$
    to $l+1$ and $\alpha'$ to $\epsilon\alpha$, and go to
    step~\ref{96ad1e}. \label{ca688f}
  \end{enumerate}
\item We are done. Any persistent homology generators not yet killed
  off are recorded as on the form $(b,\infty)$. \label{1c1792}
\end{enumerate}

The algorithm above may be summarized as follows: Compute \Cech{}
persistence until the underlying simplicial complex has at least $L$
simplices. When that is the case, walk up the complete linkage
dendrogram of the point cloud until scale $\epsilon\alpha$ is reached,
where $\alpha$ is the persistence scale. Any linkage event encountered
corresponds to an edge contraction, which is performed. After that,
computation of \Cech{} persistence resumes as before, albeit on a
reduced and changed point cloud, and collapses may happen again when
$L$ more simplices have been added. We terminate upon reaching
$\alphamax$, and ignore simplices of dimension above $D$ (thus
computing homology in dimensions $0,\dotsc,D-1$).

Note that $L$ is merely a parameter to reduce computational overhead
involved in the collapses, as a higher value postpones contractions
until the simplicial complex is denser. In principle, $L$ can be
thought of as zero. Also observe that $\epsilon=0$ corresponds to
computing ordinary \Cech{} persistence.

\subsection{Experiments}
This section describes three experiments designed to test the
feasibility of our implementation.

A calculation ranging from scale $0$ to scale $\alphamax$ will have
its resulting persistence diagram drawn as the region above the
diagonal in $[0,\alphamax]^2$. Generators still alive at $\alphamax$
will be referred to as on the form $(b,\infty)$ and plotted as
triangles, while generators of the form $(b,d)$ with $d\leq\alphamax$
will be plotted as dots. See \figref{hawaii:pd} for an example of
drawing conventions.

\subsubsection{Wedge of six circles enclosing each other} \label{sec:hawaii}

We produced a point cloud by randomly (uniformly) sampling $100$
points from a circle of radius $1$ centered at $(0,1)$, $200$ points
from a circle of radius $2$ centered at $(0,2)$ and so forth up to
$600$ points from a circle of radius $6$ centered at $(0,6)$. Each
point in the circle of radius $r$ was perturbed by radial noise
sampled from the uniform distribution on $[(1-0.02)r, (1+0.02)r)$. The
very dense region near the origin where all the circles meet (see
\figref{hawaii:cloud}) contributes nothing to homology, but
significantly adds to the number of simplices if no collapse is done.

Running to $\alphamax=2$, our implementation clearly limited the
number of simplices --- see \figref{hawaii:simplexcount} and note
especially the rapid increases between collapses, the regimes where
the ordinary \Cech{} filtration is formed --- while producing a highly
correct persistence diagram, as is shown in \figref{hawaii:pd}.

\begin{figure}[htpb]
  \centering
  \includegraphics[]{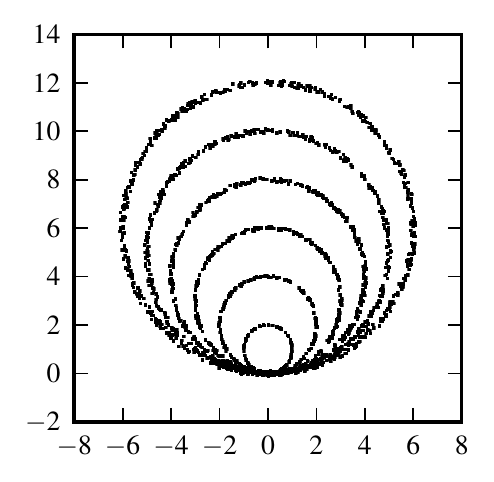}
  \caption{The point cloud from the example in \secref{hawaii}.} \label{fig:hawaii:cloud}
\end{figure}

\begin{figure}[htpb]
  \centering
  \includegraphics[]{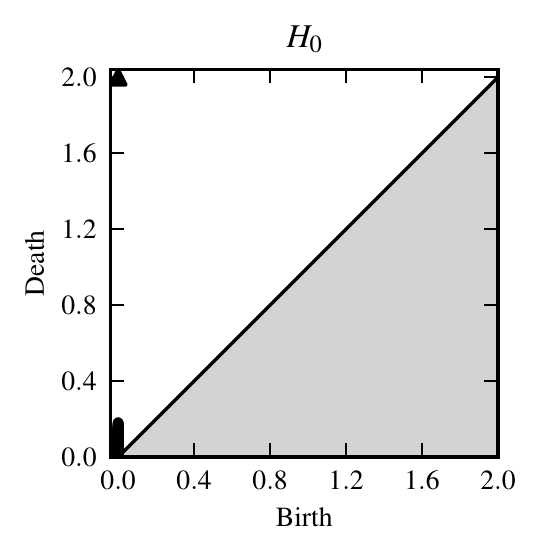}
  \includegraphics[]{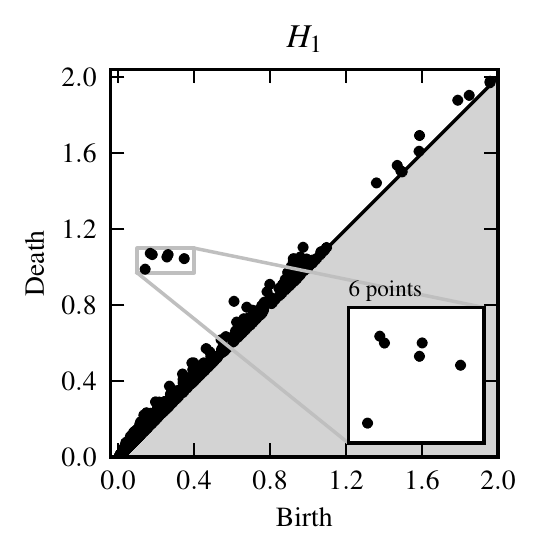}
  \caption{Persistence diagrams of the (noisy) wedge of six circles in
    \secref{hawaii} with $\epsilon=3/4$ and
    $\alphamax=2$.} \label{fig:hawaii:pd}
\end{figure}

\begin{figure}[htpb]
  \centering
  \includegraphics[]{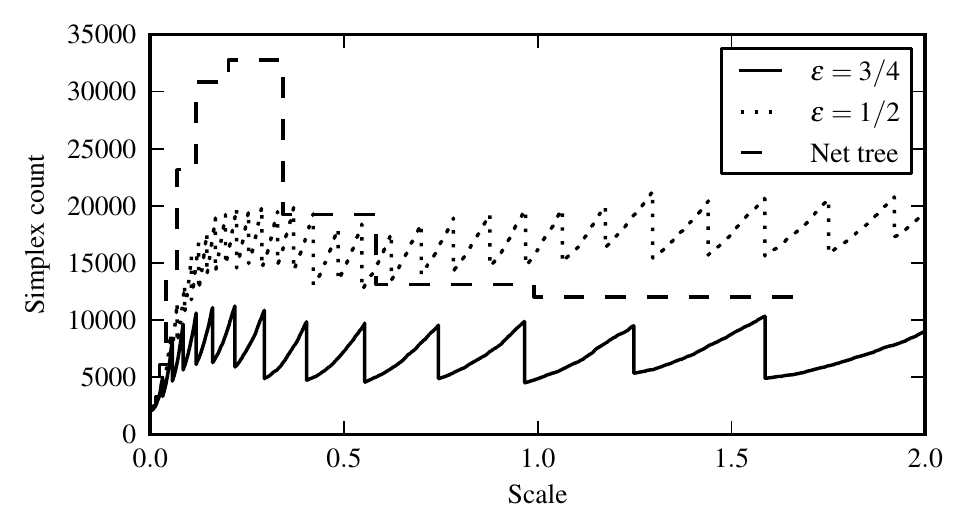}
  \caption{The simplex count while computing persistence for the
    example in \secref{hawaii}. The net tree computations were run with $\alpha_0=10^{-3}$ and $\epsilon=0.7$ in the notation of \secref{linsize}.} \label{fig:hawaii:simplexcount}
\end{figure}

\subsubsection{The real projective plane} \label{sec:rp2}
We sampled $\mathbb{R}P^2$ by randomly selecting $5000$ points on
$\mathbb{S}^2$ and embedding them in $\mathbb{R}^4$ under $(x,y,z)
\mapsto (xy, xz, y^2 -z^2, 2yz)$ as a test of how well our scheme
handles higher dimensions. \figref{rp2:pd} shows that the expected
persistence diagram resulted when computing to $\alphamax=0.54$ at
$\epsilon=1.0$.

\figref{rp2:simplexcount} compares our scheme (at $\epsilon=1$) with
the very beginning an ordinary \Cech{} filtration. Our implementation
keeps the number of simplices manageable, peaking at just above
$3\cdot10^5$ simplices near the end (scale $0.54$), while still
recovering the correct persistence diagram. The figure also shows the
simplex count for the net tree construction; notice that we were
unable to correctly choose $\alpha_0$ and $\epsilon$ so as to make
computations with it feasible, unlike for the example in
\secref{hawaii}.

\begin{figure}[htpb]
  \centering
  \includegraphics[]{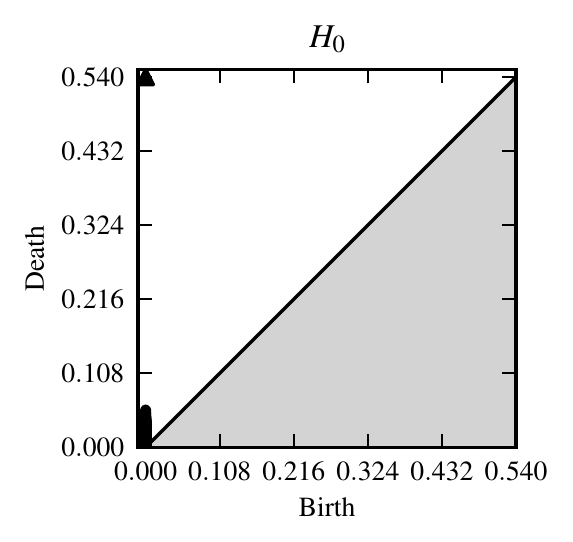}
  \includegraphics[]{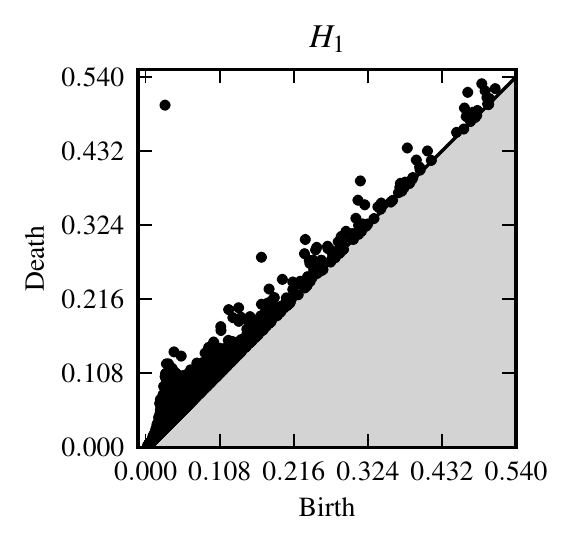}
  \includegraphics[]{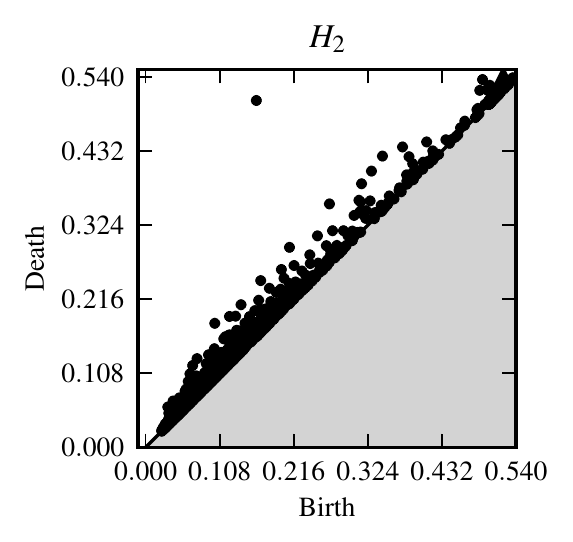}
  \includegraphics[]{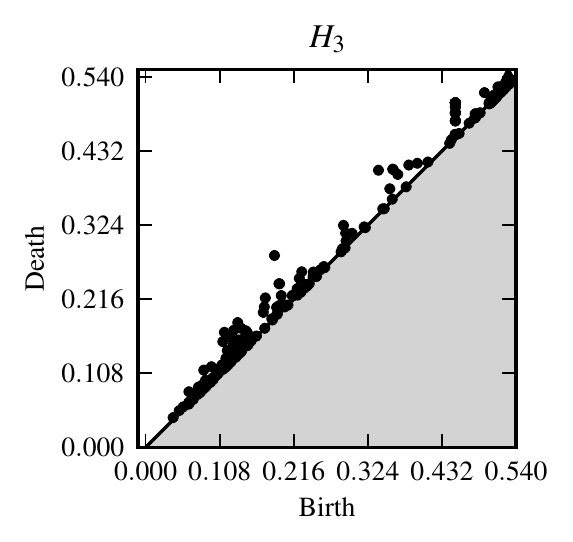}
  \caption{Persistence diagrams for the $5000$ point random sample of
    $\mathbb{R}P^2$ embedded in $\mathbb{R}^4$ as described in
    \secref{rp2}, with $\epsilon=1.0$ and
    $\alphamax=0.54$.} \label{fig:rp2:pd}
\end{figure}
\begin{figure}[htpb]
  \centering
  \includegraphics[]{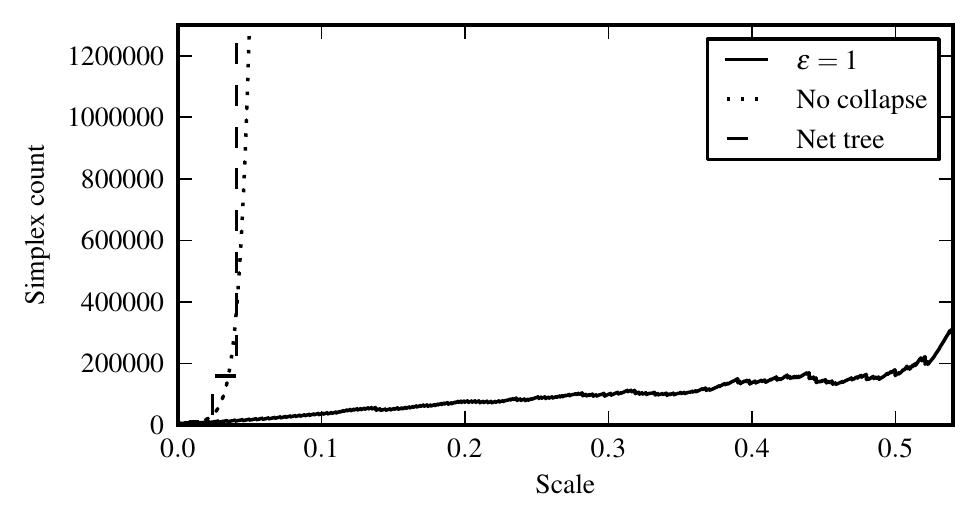}
  \caption{The simplex count for the $\mathbb{R}P^2$ example from
    \secref{rp2} compared to that of an ordinary
    \Cech{} filtration and the net tree approach (with $\alpha_0=10^{-3}$ and $\epsilon=0.7$ in the notation of \secref{linsize}).} \label{fig:rp2:simplexcount}
\end{figure}

\subsubsection{Time-delay embedding} \label{sec:timedelay} 

We solved the Lorenz system (with parameters $\sigma=10$, $r=28$,
$b=8/3$ in the notation of~\cite{lorenzatt}) and created a time series
$y\in\mathbb{R}^{15000}$ by adding together all three of the
solution's coordinates at each of $15000$ points in time. Let $A(i)$
denote the (discrete) correlation of $y$ and $y$ shifted $i$ places to
the right. The first local minimum of $A$ occurs at $15$, so that was
used as delay to embed $y$ in $\mathbb{R}^3$ by delay-embedding. The
resulting point cloud, with $15000-(3-1)\cdot 15 = 14970$ points,
reconstructs~\cite{takens1981} the Lorenz attractor as seen in
\figref{lorenz:cloud}. Observe that there are regions that have a very
high density of points.

Our implementation computes the expected persistence diagram
(\figref{lorenz:pd}) while keeping the number of simplices low
(\figref{lorenz:simplexcount}).

\begin{figure}[htpb]
  \centering
  \includegraphics[]{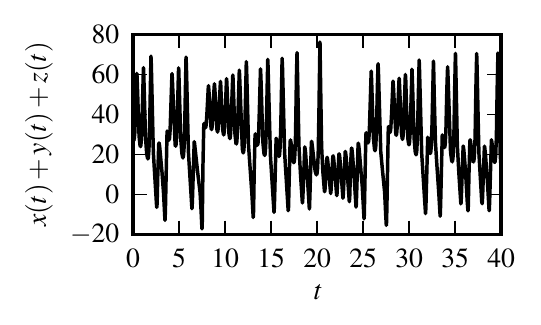}
  \includegraphics[]{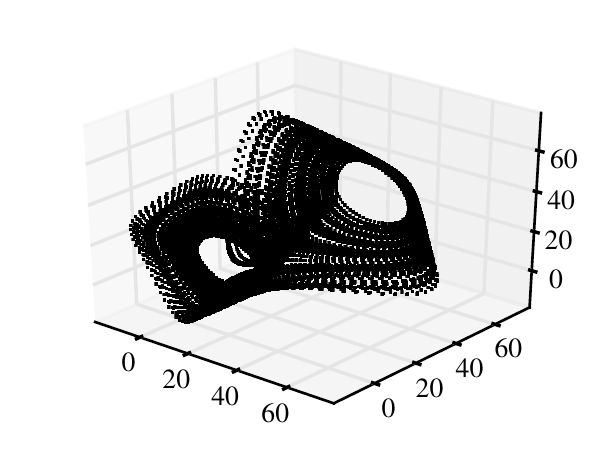}
  \caption{Lorenz system scalar measurements (parts shown on the left)
    and delay-embedding reconstructed attractor (right), as detailed
    in \secref{timedelay}.} \label{fig:lorenz:cloud}
\end{figure}
\begin{figure}[htpb]
  \centering
  \includegraphics[]{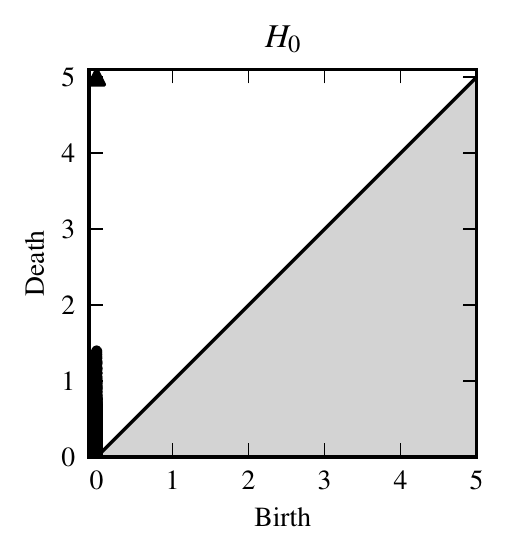}
  \includegraphics[]{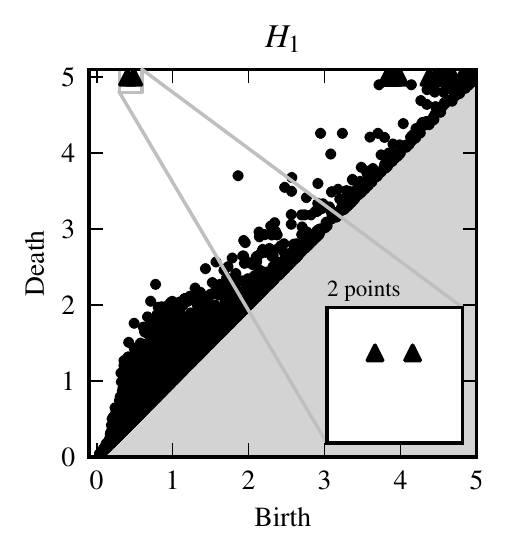}
  \caption{Persistence diagrams for the Lorenz attractor described in \secref{timedelay}.} \label{fig:lorenz:pd}
\end{figure}
\begin{figure}[htpb]
  \centering
  \includegraphics[]{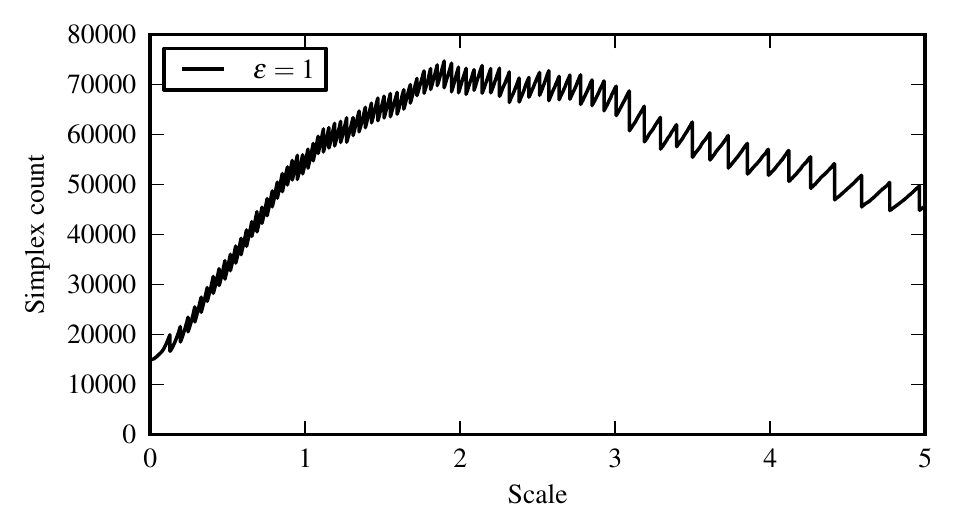}
  \caption{Simplex count for the Lorenz attractor computations described in \secref{timedelay}.} \label{fig:lorenz:simplexcount}
\end{figure}

\section{Conclusions and future work}
We have presented two approximation schemes for the \Cech{} filtration in Euclidean space. 
One construction uses a net-tree to build the \Cech{} complex at fewer and 
fewer simplices as we increase the scale parameter. The other approach forms a coarsening
of the \Cech{} filtration by using covering sets formed by unions of open balls. Computing $k$-intersections of
such covering sets is computationally expensive, so we approximated the persistence
module by choosing a representative at each scale. In practice we experienced far
better results with this method than the net-tree approach. This contrasts with the superior theoretical guarantees enjoyed by the net-tree construction. By approximating the \Cech{} filtration
through representatives we lose much of the theoretical guarantees, but the
frequent collapses allow for much greater maximum scales.

We believe that an interesting direction for future work is to find
other approximations than choosing a representative for each covering
set. This could be done either by choosing multiple representative
points, or by using the embedding to approximate the covering sets by
sets for which computing $k$-intersections is tractable.

The proofs in this paper also rely heavily on the notion of good covers. In general metric
spaces a cover by a union of balls may fail to be good, and the Nerve lemma is lost. It would
be interesting to see if there are similar results without this precondition. We believe
it should be so, as the net-tree construction for the Vietoris--Rips filtration 
extends to general metric spaces. 

\printbibliography

\end{document}